\documentclass[12pt,a4paper]{amsart}
\usepackage{fullpage,setspace}
\usepackage[T1]{fontenc}
\usepackage[utf8]{inputenc}
\usepackage{lmodern}
\usepackage{microtype}
\usepackage{mathtools,amssymb,amsthm}
\usepackage{enumitem}
\usepackage{needspace}
\usepackage{tikz}
\usepackage{tikz-cd}
\usetikzlibrary{arrows.meta}
\usepackage[numbers,sort&compress]{natbib}
\definecolor{dark-red}{rgb}{0.5,0.15,0.15}
\usepackage[colorlinks=true,linkcolor=black,citecolor=dark-red,urlcolor=dark-red]{hyperref}
\swapnumbers

\makeatletter
\let\leq\@undefined
\let\geq\@undefined
\let\vec\@undefined
\let\phi\@undefined
\let\epsilon\@undefined
\let\injlim\@undefined
\let\projlim\@undefined
\makeatother
\newcommand{\leq}{\leqslant}
\newcommand{\geq}{\geqslant}
\newcommand{\vec}{\overrightarrow}
\newcommand{\phi}{\varphi}
\newcommand{\epsilon}{\varepsilon}
\newcommand{\injlim}{\varinjlim}
\newcommand{\projlim}{\varprojlim}

\newcommand{\I}{\mathbb I}
\newcommand{\TopDH}{\mathsf{Top}_{\Delta H}}
\newcommand{\TopD}{\mathsf{Top}_{\Delta}}
\newcommand{\Cont}{\mathsf{TOP}}
\newcommand{\Rep}{\mathsf{Rep}_+}

\newcommand{\dP}{\overrightarrow{\mathcal P}}
\newcommand{\dT}{\overrightarrow{\mathcal T}}

\newcommand{\Sp}{\overrightarrow{\mathrm{Sp}}}
\newcommand{\Om}{\overrightarrow{\Omega}}
\newcommand{\Nat}{\mathcal N}
\newcommand{\im}{\operatorname{Im}}

\newcommand{\id}{\operatorname{id}}
\newcommand{\co}{\mathrm{co}}

\theoremstyle{plain}
\newtheorem{theorem}{Theorem}[section]
\newtheorem{proposition}[theorem]{Proposition}
\newtheorem{lemma}[theorem]{Lemma}
\newtheorem{corollary}[theorem]{Corollary}
\newtheorem*{theorem*}{Theorem}
\newtheorem*{corollary*}{Corollary}

\theoremstyle{definition}
\newtheorem{definition}[theorem]{Definition}

\newtheorem{notation}[theorem]{Notation}
\newtheorem{remark}[theorem]{Remark}

\newcommand{\Top}{\mathbf{Top}}
\newcommand{\dTop}{\mathbf{dTop}}
\newcommand{\sdTop}{\mathbf{sdTop}}
\newcommand{\Timed}{\mathbf{Timed}}
\newcommand{\Glob}{\operatorname{Glob}^{\mathrm{top}}}

\newcommand{\dS}{\vec{\mathsf{S}}^{1}}
\newcommand{\Cg}{\mathsf C_{\mathrm{gl}}}

\title{Saturated directed spaces and generalized clocks}
\hypersetup{pdftitle={Saturated directed spaces and generalized clocks},pdfauthor={Philippe Gaucher}}
\author[P. Gaucher]{Philippe Gaucher}
\address{Universit\'e Paris Cit\'e, CNRS, IRIF, F-75013, Paris, France}
\urladdr{\url{https://www.irif.fr/~gaucher}}
\subjclass[2020]{Primary 55P10; Secondary 18C35, 54C35, 68Q85}
\keywords{directed topology, timed space, clock, trace space, saturation, local presentability, globular complex}

\begin{document}
\begin{abstract}
We construct a Hausdorff saturated directed space whose directed path space between two distinct points has nontrivial fundamental group although the corresponding trace space is a square. A clock is a directed space with submetrizable underlying space; a timed space is a saturated directed space equipped with a regular directed map to a clock. Over every fixed clock, timed spaces form a locally presentable category, and their quotients from paths to traces are trivial Hurewicz fibrations with continuous sections. For Hausdorff saturated directed spaces whose compact subspaces are metrizable, these quotients are trivial q-fibrations. We construct a globular clock containing the directed circle and admitting regular maps from all q-cofibrant globular realizations, whereas a finite cellular example admits no regular map to the circle. Normalization also gives a proof of the quotient theorem for execution paths and establishes the homotopy lifting property for nondecreasing reparametrizations.
\end{abstract}
\maketitle
\setcounter{tocdepth}{1}
\tableofcontents
\hypersetup{linkcolor=dark-red}

\section{Introduction}

\subsection*{Presentation} A concurrent program can be represented geometrically by a space of states. A point is a state of the program, and a path is an execution. For example, two coordinates may record the progress of two actions: a square represents their independent execution, whereas restrictions on simultaneous access to a resource may exclude part of that square (Figure~\ref{fig:concurrent-executions}). The order of the actions is essential: an execution cannot in general be reversed. Directed algebraic topology studies such spaces together with distinguished paths that express the allowed direction of execution; see \cite{DAT_book,Grandis}.

\begin{figure}[htbp]
\centering
\begin{tikzpicture}[x=1cm,y=1cm,>=stealth,font=\small]
 \begin{scope}
  \node[align=center] at (1.7,3.9) {Independent actions};
  \fill[gray!5] (0,0) rectangle (3.4,3.4);
  \draw[gray] (0,0) rectangle (3.4,3.4);
  \draw[->] (-.12,0)--(3.8,0) node[right] {$t_1$};
  \draw[->] (0,-.12)--(0,3.7) node[left] {$t_2$};
  \draw[gray!70,dashed] (0,0)--(3.4,3.4);
  \draw[gray!70,dashed] (0,0)--(.8,1.7)--(2.6,2.6)--(3.4,3.4);
  \draw[thick,->] (0,0)--(1.7,.6)--(2.7,1.9)--(3.4,3.4);
  \draw[thick,dark-red,->] (0,0)--(.6,1.7)--(1.9,2.7)--(3.4,3.4);
  \fill (0,0) circle (1.6pt) node[below left] {$u$};
  \fill (3.4,3.4) circle (1.6pt) node[above right] {$v$};
 \end{scope}
 \begin{scope}[xshift=6.3cm]
  \node[align=center] at (1.7,3.9) {A shared resource};
  \fill[gray!5] (0,0) rectangle (3.4,3.4);
  \draw[gray] (0,0) rectangle (3.4,3.4);
  \draw[->] (-.12,0)--(3.8,0) node[right] {$t_1$};
  \draw[->] (0,-.12)--(0,3.7) node[left] {$t_2$};
  \fill[gray!25] (1.1,1.1) rectangle (2.3,2.3);
  \draw[gray,dashed] (1.1,1.1) rectangle (2.3,2.3);
  \node[align=center,font=\scriptsize] at (1.7,1.7) {forbidden\\interior};
  \draw[thick,->] (0,0)--(2.7,.65)--(2.95,2.7)--(3.4,3.4);
  \draw[thick,dark-red,->] (0,0)--(.65,2.7)--(2.7,2.95)--(3.4,3.4);
  \fill (0,0) circle (1.6pt) node[below left] {$u$};
  \fill (3.4,3.4) circle (1.6pt) node[above right] {$v$};
 \end{scope}
\end{tikzpicture}
\caption{Coordinates describe the progress of two actions. Along a directed path, both coordinates are nondecreasing. In the square on the left, the displayed executions can be deformed into one another through directed paths. After removing the open central rectangle on the right, the two displayed paths belong to distinct path components of the directed path space from $u$ to $v$.}
\label{fig:concurrent-executions}
\end{figure}

Grandis' directed spaces specify sets of continuous paths containing all constant paths and closed under nondecreasing changes of parameter and concatenation. The \emph{trace} of a path is its equivalence class under nondecreasing surjective reparametrizations. It records the order of execution while forgetting changes of speed and stop intervals. When speed is not part of the computational information, one would like the path and trace spaces to have the same homotopy type.

The axioms of directed spaces do not ensure this. In Section~\ref{ce:sec:counterexample}, we construct a saturated directed space with Hausdorff $\Delta$-generated underlying space whose trace space between two distinct points is a square, whereas its directed path space has nontrivial fundamental group (Theorem~\ref{ce:thm:main}). Saturation means that removing stop intervals preserves directed paths. Thus saturation and the Hausdorff condition do not suffice for the quotient to be a weak homotopy equivalence. The underlying space is not metrizable, and the quotient is not a quasifibration either (Remark~\ref{ce:rem:no-clock} and Corollary~\ref{ce:cor:not-quasifibration}). The same example applies to regular directed paths modulo increasing homeomorphisms.

The paper \cite{clockmap} equips a saturated directed space with a regular map to the directed circle, called the \emph{cubical clock} here. Regularity means that a directed path whose image under the map is constant is itself constant. These maps form a locally presentable category, and their quotients from paths to traces are homotopy equivalences. The construction extends Raussen's cubical naturalization: the sum of the coordinates measures progress in a cube, and reduction modulo integers makes these functions agree on faces \cite{MR2521708}.

The cubical clock does not suffice for globular realizations. Proposition~\ref{prop:no-cubical-clock} constructs a finite cellular multipointed $d$-space with an execution loop $a$ homotopic to $a*_N a$. A regular map to the circle would give this loop a positive winding number $n$ satisfying $n=2n$. We therefore need a more general notion of clock.

A \emph{clock} is a directed space whose underlying space is submetrizable, that is, admits a continuous injective map into a metric space. A \emph{timed space} over a clock $C$ is a saturated directed space $X$ equipped with a regular directed map $p:X\to C$. Morphisms commute with the clock maps. We seek convenient categories for directed topology in two respects: local presentability, for subsequent constructions of model categories, and preservation of homotopy type when passing from paths to traces.

The following theorem gives the homotopy lifting property for the quotients from paths to traces of timed spaces.

\begin{theorem*}
For every clock $C$, the category $\Timed(C)$ of timed spaces over $C$ is locally presentable. For every object $p:X\to C$ and every $u,v\in|X|$, the quotient map
\[
q_{u,v}:\dP(X)(u,v)\longrightarrow\dT(X)(u,v)
\]
is a trivial Hurewicz fibration, that is, a Hurewicz fibration and a homotopy equivalence. In particular, it is a trivial q-fibration. The trace space is homeomorphic to a strong deformation retract of the directed path space.
\end{theorem*}

Every homotopy of traces therefore admits a continuous family of representatives extending a prescribed initial family. Once a continuous injection of the clock into a metric space and the progress function are fixed, normalization and the homotopy lifts are natural for morphisms over that clock.

For a Hausdorff saturated directed space whose compact subspaces are metrizable, the quotients from paths to traces are trivial q-fibrations and the trace spaces are $\Delta$-Hausdorff (Theorem~\ref{thm:metrizable-quasicompacts-traces}). The proof applies the clock construction to the compact subspaces determined by compact families of traces. In particular, the result holds for Hausdorff second countable spaces (Theorem~\ref{thm:second-countable-traces}).

Fix a cellular q-cofibrant replacement $\mathbf1^{\mathrm{cell}}\to\mathbf1$ of the final $\mathcal M$-multipointed $d$-space. Its underlying directed space
\[
\Cg=\Sp(\mathbf1^{\mathrm{cell}})
\]
is the \emph{globular clock}. Every q-cofibrant globular realization is a saturated clock and admits a regular map to this fixed $\Cg$. An injective directed map $\dS\to\Cg$ gives a full and faithful functor
\[
\Timed(\dS)\longrightarrow\Timed(\Cg).
\]
Thus $\Timed(\Cg)$ contains both cubical and globular realizations. Restricting normalization to execution paths proves the homotopy equivalence in \cite[Theorem~16]{Moore3}\footnote{A flaw in \cite{Moore3} invalidates its proof of the Quillen equivalence: the unit of the Quillen adjunction need not be an isomorphism on q-cofibrant Moore $\mathcal M$-flows. The Quillen equivalence remains valid \cite[Theorem~25.4]{MooreCorrections-full}. The correction uses Proposition~\ref{prop:cellular-submetrizable} and the normalization results of Sections~\ref{sec:metric}--\ref{sec:normalization}, all proved independently of it. None of the false assertions of \cite{Moore3} is used in this paper. The results of \cite{GlobularNaturalSystem-full} used here are valid and do not depend on those assertions.} for both reparametrization categories and establishes the homotopy lifting property for the quotient by nondecreasing surjections.

The categorical argument extends the description as a small-orthogonality class in \cite{clockmap}. The topological argument uses truncated variation to construct a continuous progress function with the same stop intervals as the path. Factoring through this function gives normalization and a section of the quotient. Preservation of images controls continuity at constant paths. For the globular clock, we extend families of execution paths along cellular attachments and make the extension regular on each interior meridian. The result passes to retracts.

\subsection*{Outline of the paper} Section~\ref{sec:convenient} fixes the categories of spaces, path and trace topologies, saturation, and regularity. Section~\ref{ce:sec:counterexample} constructs the counterexample, identifies its trace space with a square, and proves that its directed path space has a nontrivial fundamental group; it also treats regular paths and the ordinary compact-open topology. Section~\ref{sec:timed-category} proves local presentability of $\Timed(C)$; this argument does not require submetrizability. Section~\ref{sec:metric} constructs the metric progress function from truncated variations. Section~\ref{sec:normalization} proves continuity of normalization in the source topology, identifies traces with normalized paths, and establishes the homotopy lifting property. It also treats invariant families of paths. Section~\ref{sec:cubical} applies the construction to precubical sets. Section~\ref{sec:globular} constructs the fixed globular clock, compares it with the cubical clock, and proves the quotient theorem for execution paths. Section~\ref{sec:without-global-clocks} treats Hausdorff saturated directed spaces with metrizable compact subspaces by restricting to compact families of traces. Section~\ref{sec:functorial-normalization} proves naturality over a fixed clock and describes the natural splitting, the category of normalized paths, and the action of the multiplicative monoid $\I$.

\subsection*{Acknowledgments}

The use of truncated variations in Section~\ref{sec:metric} arose from discussions with GPT-6 Astra. It replaces the length function in Raussen's naturalization \cite{MR2521708}; the function $x\mapsto L_x$ in \eqref{eq:multiscale-metric-clock} is not a length function. GPT-6 Astra also suggested proving the Hurewicz fibration property when a global clock exists. The argument using compact families in Section~\ref{sec:without-global-clocks} addresses the original question of working with local clocks. The paper adapts and generalizes techniques and results of \cite{clockmap,MR2521708}.

\section{Spaces, directed paths, and traces}
\label{sec:convenient}

We write $\I=[0,1]$, $\mathsf D^n$ for the closed unit disk in $\mathbb R^n$, and $\mathsf S^{n-1}=\partial\mathsf D^n$, with $\mathsf D^0=\{0\}$ and $\mathsf S^{-1}=\varnothing$. Throughout, $\Top$ denotes either the category $\TopD$ of $\Delta$-generated spaces or its full subcategory $\TopDH$ of $\Delta$-Hausdorff $\Delta$-generated spaces. A space $Y$ is $\Delta$-Hausdorff if the image of every continuous map $\I\to Y$ is closed. Both categories are locally presentable and cartesian closed; see \cite[Corollary~3.7]{FR-full} and \cite[Section~2 and Appendix~B]{leftproperflow-full}. Statements apply to both categories unless otherwise specified.

For an ordinary topological space $Y$, its $\Delta$-kelleyfication $k_\Delta Y$ has the final topology with respect to continuous maps from standard topological simplices to $Y$. We write
\[
\Cont_\Delta(A,B)=k_\Delta\Cont_\co(A,B),
\qquad A\times_\Delta B=k_\Delta(A\times B),
\]
where $\Cont_\co(A,B)$ has the compact-open topology. Relative topologies are taken before $\Delta$-kelleyfication. In particular, an arbitrary subspace of a $\Delta$-generated space is not assumed to be $\Delta$-generated.

An \emph{embedding} $i:A\hookrightarrow B$ induces a homeomorphism from $A$ onto $i(A)$ with the relative topology. It is a \emph{closed embedding} when $i(A)$ is closed in $B$.

We use the q-model structure on $\Top$, whose weak equivalences are weak homotopy equivalences; see \cite[Section~2 and Appendix~B]{leftproperflow-full}. For $p:E\to B$, the q-fibration and trivial q-fibration conditions are, respectively, the existence of a lift in every square of the following forms, for all $n\geq0$:
\[
\begin{tikzcd}[column sep=large,row sep=large]
\mathsf D^n\times\{0\}\arrow[r]\arrow[d,hook]&E\arrow[d,"p"]\\
\mathsf D^n\times\I\arrow[r]\arrow[ur,dashed]&B
\end{tikzcd}
\qquad
\begin{tikzcd}[column sep=large,row sep=large]
\mathsf S^{n-1}\arrow[r]\arrow[d,hook]&E\arrow[d,"p"]\\
\mathsf D^n\arrow[r]\arrow[ur,dashed]&B.
\end{tikzcd}
\]
The inclusions $\mathsf S^{n-1}\hookrightarrow\mathsf D^n$ generate the q-cofibrations.

\begin{lemma}[well known]
\label{lem:interval-product}
If $Z$ is $\Delta$-generated, the ordinary product $Z\times\I$ is $\Delta$-generated. Consequently the homotopies used below are ordinary continuous homotopies as well as homotopies in $\Top$.
\end{lemma}
\begin{proof}
The defining final topology gives a quotient
\[
\coprod_{n\geq0}\coprod_{\alpha:\Delta^n\to Z}\Delta^n\longrightarrow Z.
\]
Product with the locally compact Hausdorff space $\I$ preserves quotient maps \cite[\S46, Exercise~9]{zbMATH01461253}. Each $\Delta^n\times\I$ admits a finite triangulation, so
\[
\coprod_{n\geq0}\coprod_{\alpha:\Delta^n\to Z}(\Delta^n\times\I)
\longrightarrow Z\times\I
\]
is a quotient of a $\Delta$-generated space. Final quotients preserve $\Delta$-generation.
\end{proof}

\begin{definition}
A \emph{directed space} $X=(|X|,dX)$ consists of a space $|X|\in\Top$ and a set $dX$ of continuous maps $\I\to|X|$ containing all constant paths and closed under nondecreasing continuous reparametrizations and normalized composition
\[
(x*_Ny)(t)=
\begin{cases}x(2t),&0\leq t\leq\tfrac12,\\
y(2t-1),&\tfrac12\leq t\leq1,
\end{cases}
\quad x(1)=y(0).
\]
A morphism preserves directed paths. The category of directed spaces is denoted by $\dTop$.
\end{definition}

The directed interval $\vec\I$ has underlying space $\I$ and all continuous nondecreasing paths.

Let $\Rep$ be the monoid of continuous nondecreasing surjections $\I\to\I$, with its uniform, equivalently compact-open, topology. Its elements fix both endpoints. It is a convex, metrizable, locally path-connected space, hence is $\Delta$-generated.

\begin{definition}
A directed space is \emph{saturated} if, for every continuous $x:\I\to|X|$ and every $\phi\in\Rep$, the condition $x\circ\phi\in dX$ implies $x\in dX$. We denote the full subcategory of saturated directed spaces by $\sdTop$. A morphism $p:X\to C$ is \emph{regular} if $p\circ x$ constant implies $x$ constant for every $x\in dX$.
\end{definition}

These are the definitions of \cite[Section~2]{clockmap}. A \emph{regular path} is either constant or has no constant nondegenerate subinterval; see \cite[Definition~1.1(3)]{reparam}.

\begin{lemma}
\label{lem:regular-stops}
If $p:X\to C$ is regular and $x$ is directed, then
\[
(p\circ x)|_{[a,b]}\text{ is constant}
\quad\Longleftrightarrow\quad
x|_{[a,b]}\text{ is constant}
\]
for every $0\leq a\leq b\leq1$.
\end{lemma}
\begin{proof}
For $a<b$, apply regularity to the path $t\mapsto x(a+(b-a)t)$. The converse and the case $a=b$ are immediate.
\end{proof}

\begin{lemma}
\label{lem:regular-generators}
Suppose that a family $\mathcal A\subseteq dX$ generates the directed structure of $X$ under nondecreasing reparametrizations, concatenation, and addition of constant paths. A directed map $p:X\to C$ is regular if and only if, for every $x\in\mathcal A$ and $0\leq a\leq b\leq1$,
\[
(p\circ x)|_{[a,b]}\text{ constant}
\quad\Longrightarrow\quad x|_{[a,b]}\text{ constant}.
\]
\end{lemma}
\begin{proof}
Necessity follows from Lemma~\ref{lem:regular-stops}. For sufficiency, let $\mathcal E\subseteq dX$ be the paths satisfying the displayed implication on every subinterval. For continuous nondecreasing $\phi:\I\to\I$,
\[
c_z\in\mathcal E,\qquad
x\in\mathcal E\Longrightarrow x\circ\phi\in\mathcal E,\qquad
x,y\in\mathcal E,\ x(1)=y(0)\Longrightarrow x*_Ny\in\mathcal E.
\]
The second implication uses $\phi([a,b])=[\phi(a),\phi(b)]$; the third splits $[a,b]$ at $1/2$. Hence $\mathcal A\subseteq\mathcal E$ implies $dX=\mathcal E$. Taking $[a,b]=\I$ proves regularity.
\end{proof}

For $u,v\in|X|$, let
\begin{align*}
\dP_\co(X)(u,v)=\{x\in dX\mid x(0)=u,\ x(1)=v\}
&\subseteq\Cont_\co(\I,|X|),\\
\dP(X)(u,v)&=k_\Delta\dP_\co(X)(u,v).
\end{align*}
Trace equivalence is the equivalence relation generated by $x\sim x\circ\phi$ for $\phi\in\Rep$. We use the same relation on all continuous paths when no directed structure is specified. Equivalent paths have the same image and endpoints. The \emph{trace space} $\dT(X)(u,v)$ is the set of equivalence classes with the ordinary quotient topology of $\dP(X)(u,v)$, and $q_{u,v}$ is the quotient map. This quotient is automatically $\Delta$-generated. Under our hypotheses it will also be $\Delta$-Hausdorff when $|X|$ is $\Delta$-Hausdorff; no reflection of the quotient is needed.

\begin{lemma}\label{lem:path-families}
Let $Z$ be an ordinary topological space, put $Y=k_\Delta Z$, and let $K$ be a finite polyhedron. Let $\mathcal A$ be a family of continuous paths $\I\to Y$, equipped with the $\Delta$-kelleyfication of its compact-open subspace topology. A set map $f:K\to\mathcal A$ is continuous if and only if its adjoint $(k,t)\mapsto f(k)(t)$ is continuous into $Z$.
\end{lemma}
\begin{proof}
Let $\mathcal A_{\co}$ denote the relative compact-open space and $\widehat f(k,t)=f(k)(t)$. Since $K$ and $K\times\I$ are $\Delta$-generated, coreflection and the exponential law \cite[Theorems~46.10--46.11]{zbMATH01461253} give
\[
\begin{aligned}
f:K\to k_\Delta\mathcal A_{\co}\text{ continuous}
&\Longleftrightarrow f:K\to\mathcal A_{\co}\text{ continuous}\\
&\Longleftrightarrow\widehat f:K\times\I\to Y\text{ continuous}\\
&\Longleftrightarrow\widehat f:K\times\I\to Z\text{ continuous}.
\end{aligned}
\]
\end{proof}

\section{A quotient by reparametrization with different weak homotopy types}
\label{ce:sec:counterexample}

We construct a saturated directed space for which the quotient from directed paths to traces fails to be a weak homotopy equivalence. Throughout this section, products and subspaces have their ordinary topologies unless $\Delta$-kelleyfication is indicated.

\begin{theorem}\label{ce:thm:main}
There exist a saturated directed space $X$ with a Hausdorff $\Delta$-generated underlying space and distinct points $u,v\in|X|$ such that $\dT(X)(u,v)$ is homeomorphic to $[-2,2]^2$ and $\dP(X)(u,v)$ has a nontrivial fundamental group at a suitable basepoint. In particular, the canonical quotient $q_{u,v}:\dP(X)(u,v)\to\dT(X)(u,v)$ is not a weak homotopy equivalence.
\end{theorem}

\subsection{A compact family of intervals}
\label{ce:sec:intervals}

Start with a family $B\times \I$ of directed intervals, indexed by the square $B=[-2,2]^2$. Replace the midpoint of each fiber indexed by a point of the unit circle $A\subset B$ by an interval. Figure~\ref{ce:fig:F-three-dimensional} depicts the resulting family $F$. Then collapse all lower endpoints to one point $u$ and all upper endpoints to another point $v$, so that the resulting intervals share exactly their two endpoints. This construction uses the method of split compact spaces in \cite[Definition~2.1]{Koszmider16-full}: the points of $A\times\{1/2\}\subset B\times \I$ are replaced by intervals.

Let $A=\{a\in\mathbb R^2:\|a\|=1\}$, where $\|\cdot\|$ denotes the Euclidean norm. Write $s$ for the height coordinate in $B\times \I$. For each $a\in A$, an additional coordinate $t_a$ records progress along the interval inserted at $(a,1/2)$. Define the continuous nondecreasing surjection
\begin{equation}\label{ce:eq:plateau}
c(t)=
\begin{cases}
\dfrac32t,&0\leq t\leq\dfrac13,\\[7pt]
\dfrac12,&\dfrac13\leq t\leq\dfrac23,\\[7pt]
\dfrac{3t-1}{2},&\dfrac23\leq t\leq1,
\end{cases}
\end{equation}
and, for every $a\in A$, put
\begin{equation}\label{ce:eq:ga}
\delta_a(b)=\min\{1,\|b-a\|\},
\qquad
 g_a(b,t)=\delta_a(b)t+(1-\delta_a(b))c(t).
\end{equation}
For $b\ne a$ and $t<t'$, monotonicity of $c$ gives
\[
\begin{aligned}
g_a(b,t')-g_a(b,t)
&=\delta_a(b)(t'-t)+(1-\delta_a(b))(c(t')-c(t))\\
&\geq\delta_a(b)(t'-t)>0.
\end{aligned}
\]
Since $g_a(b,0)=0$ and $g_a(b,1)=1$, the map $g_a(b,-)$ is then an increasing homeomorphism of $\I$. For $b=a$ it equals $c$, whose inverse image of $1/2$ is $[1/3,2/3]$. Thus the equation $g_a(b,t_a)=s$ determines $t_a$ uniquely except at $(b,s)=(a,1/2)$, where it supplies the inserted interval.

Set
\begin{equation}\label{ce:eq:F}
F=\left\{(b,s,(t_a)_{a\in A})\in B\times \I\times \I^A:
 g_a(b,t_a)=s\text{ for every }a\in A\right\},
\end{equation}
where $\I^A$ has the product topology, and let $p:F\to B$ be the projection. Each equation defines a closed set, so $F$ is a closed subspace of the compact product $B\times\I\times\I^A$ and is compact by Tychonoff's theorem \cite[Theorem~37.3]{zbMATH01461253}. Each fiber $F_b=p^{-1}(\{b\})$ is compact.

\begin{remark}\label{ce:rem:exceptional-set}
The construction and all conclusions of this section remain valid if the unit circle is replaced by any uncountable subset $A$ of the closed unit disk $\{b\in\mathbb R^2:\|b\|\leq1\}$ such that $A\cap S$ is at most countable for every straight segment $S$. Equivalently, $\alpha(\I)\cap A$ is at most countable for every piecewise linear path $\alpha:\I\to B$. The curve in Lemma~\ref{ce:lem:quotient-curves} has image a countable union of straight segments and one limit point, so its intersection with $A$ remains countable and Lemma~\ref{ce:lem:curve-lifting} still applies. Uncountability of $A$ gives the obstruction in Lemma~\ref{ce:lem:countability-obstruction}. The bound $\|a\|\leq1$ preserves the boundary family in \eqref{ce:eq:boundary-loop}. The set $A$ need not be closed.
\end{remark}

\begin{figure}[htbp]
	\centering
	\begin{tikzpicture}[
		x={(.92cm,0cm)}, y={(.40cm,.28cm)}, z={(0cm,4.3cm)},
		>=Stealth, font=\small,
		edge/.style={gray!60,thin},
		ordinary/.style={teal!70!black,thick},
		exceptional/.style={dark-red!65!black,thin},
		inserted/.style={dark-red!75!black,line width=2.5pt}
		]
		
		% The vertical drawing coordinate indicates fiber order, not s.
		\fill[gray!7]
		(-2,-2,0)--(2,-2,0)--(2,2,0)--(-2,2,0)--cycle;
		\draw[edge]
		(-2,-2,0)--(2,-2,0)--(2,2,0)--(-2,2,0)--cycle;
		\foreach \x/\y in {-2/-2,2/-2,2/2,-2/2}
		\draw[edge,densely dashed] (\x,\y,0)--(\x,\y,1);
		\draw[edge]
		(-2,-2,1)--(2,-2,1)--(2,2,1)--(-2,2,1)--cycle;
		
		% Indexing circle in the lower endpoint slice.
		\draw[dark-red!65!black,densely dashed]
		plot[domain=0:360,samples=81,smooth]
		({cos(\x)},{sin(\x)},0);
		\node[left,dark-red!65!black] at (-1,0,0) {$A$};
		\node[below,yshift=-10pt] at (0,-2,0)
		{$E_-\cong B=[-2,2]^2$};
		\node[above] at (0,2,1) {$E_+\cong B$};
		
		% Representative exceptional fibers.
		% Each thick middle interval has its own coordinate t_a.
		\foreach \angle in {72,144,0,216,288}{
			\draw[exceptional]
			({cos(\angle)},{sin(\angle)},0)--
			({cos(\angle)},{sin(\angle)},1);
			\draw[inserted]
			({cos(\angle)},{sin(\angle)},1/3)--
			({cos(\angle)},{sin(\angle)},2/3);
			\draw[exceptional,->]
			({cos(\angle)},{sin(\angle)},.77)--
			({cos(\angle)},{sin(\angle)},.90);
		}
		\fill[dark-red!75!black] (1,0,0) circle (1.5pt);
		\node[below right,dark-red!65!black] at (1,0,0) {$a$};
		\node[right,dark-red!65!black] at (1,0,.88) {$F_a$};
		
		% An ordinary fiber: its s=1/2 slice is a single point.
		\draw[ordinary] (-1.5,-1.5,0)--(-1.5,-1.5,1);
		\draw[ordinary,->] (-1.5,-1.5,.72)--(-1.5,-1.5,.88);
		\fill[teal!70!black] (-1.5,-1.5,.5) circle (2pt);
		\fill[teal!70!black] (-1.5,-1.5,0) circle (1.5pt);
		\node[below left,teal!70!black] at (-1.5,-1.5,0)
		{$b\notin A$};
		\node[left,teal!70!black] at (-1.5,-1.5,.83) {$F_b$};
		\node[left,teal!70!black] at (-1.5,-1.5,.5)
		{$s=\tfrac12$};
		
		% Exact coordinate picture of an exceptional fiber.
		\begin{scope}[shift={(4cm,.65cm)},x=3.2cm,y=3.2cm]
			\node[above] at (.5,1.13) {$F_a:\quad s=c(t_a)$};
			\draw[->] (0,0)--(1.14,0) node[right] {$t_a$};
			\draw[->] (0,0)--(0,1.10) node[above] {$s$};
			
			\draw[gray,densely dotted] (0,.5)--(1/3,.5);
			\foreach \t in {1/3,2/3}
			\draw[gray,densely dotted] (\t,0)--(\t,.5);
			
			\draw[exceptional,thick]
			(0,0)--(1/3,.5)
			(2/3,.5)--(1,1);
			\draw[inserted] (1/3,.5)--(2/3,.5);
			\draw[dark-red!75!black,->,thick] (.43,.5)--(.59,.5);
			\node[above,dark-red!75!black] at (.5,.53) {$I_a$};
			
			\node[left] at (0,.5) {$\tfrac12$};
			\node[left] at (0,1) {$1$};
			\node[below left] at (0,0) {$0$};
			\node[below] at (1/3,0) {$\tfrac13$};
			\node[below] at (2/3,0) {$\tfrac23$};
			\node[below] at (1,0) {$1$};
			
			\node[below,align=center] at (.5,-.20)
			{$I_a=F_a\cap\{s=\tfrac12\}$\\[3pt]
				$t_a\in[\tfrac13,\tfrac23],\quad
				t_{a'}=\tfrac12\ (a'\ne a)$};
		\end{scope}
	\end{tikzpicture}
	
\caption{Schematic of $F$, with a few fibers displayed and endpoint slices $E_-=\{s=0\}$ and $E_+=\{s=1\}$. For $b\notin A$, $s$ parametrizes $F_b$. For $a\in A$, the thick segment $I_a$ is inserted at $s=\tfrac12$ and parametrized by $t_a\in[\tfrac13,\tfrac23]$. Vertical position in the left panel indicates fiber order only; the drawing is not an embedding of $F$ into $\mathbb R^3$.}
		\label{ce:fig:F-three-dimensional}
\end{figure}

\begin{lemma}\label{ce:lem:fibers}
Each fiber $F_b$ is homeomorphic to $\I$, ordered by $s$ when $b\notin A$ and by $t_b$ when $b\in A$. Its lower endpoint has $s=t_a=0$ for all $a$, and its upper endpoint has $s=t_a=1$ for all $a$. All the coordinates $s,t_a$ are nondecreasing along the fiber.
\end{lemma}
\begin{proof}
The following continuous coordinate maps are bijections:
\[
\begin{array}{c|c|c}
 &F_b\longrightarrow\I&\text{remaining coordinates}\\[3pt]\hline
 b\notin A&s&t_a=g_a(b,-)^{-1}(s)\quad(a\in A)\\[3pt]
 b=a_0\in A&t_{a_0}&s=c(t_{a_0}),\quad t_a=g_a(b,-)^{-1}(s)\quad(a\ne a_0).
\end{array}
\]
They are homeomorphisms by compactness \cite[Theorem~26.6]{zbMATH01461253}. The displayed coordinates are nondecreasing, and
\[
g_a(b,-)^{-1}(\{0\})=\{0\},\qquad g_a(b,-)^{-1}(\{1\})=\{1\}
\]
gives the endpoints.
\end{proof}

Write $\{s=h\}$ for the slice of $F$ at height $h\in \I$. The endpoint slices $E_-:=\{s=0\}$ and $E_+:=\{s=1\}$ are both homeomorphic to $B$ under $p$. Collapse $E_-$ to $u$ and $E_+$ to $v$, leaving all other points unchanged, and denote the quotient by $\rho:F\to Z$. The following square is a pushout in ordinary topological spaces:
\[
\begin{tikzcd}[column sep=large,row sep=large]
 E_-\sqcup E_+\arrow[r,hook]\arrow[d,two heads]&F\arrow[d,two heads,"\rho"]\\
 \{u,v\}\arrow[r,hook]&Z\arrow[ul,phantom,"\ulcorner",very near start].
\end{tikzcd}
\]
The relation $R$ and the saturation of a closed subset $K\subset F$ are
\[
R=\operatorname{diag}_F\cup(E_-\times E_-)\cup(E_+\times E_+),\qquad
\rho^{-1}\rho(K)=\operatorname{pr}_2\bigl(R\cap(K\times F)\bigr).
\]
The relation is closed, and the displayed projection is closed by compactness. Thus $\rho$ is closed with compact fibers, and $Z$ is Hausdorff by \cite[\S26, Exercise~7, and \S31, Exercise~7(a)]{zbMATH01461253}; hence $Z$ is compact. The coordinates $s,t_a$ descend continuously to $Z$, being constant on each collapsed slice. We retain the notation $\{s=h\}$ for its slices.

Put $X=k_\Delta Z$ and $J_b=\rho(F_b)\subset|X|$. The topology of $X$ refines that of the Hausdorff space $Z$, so $X$ is Hausdorff. A parametrization $\I\to F_b\to Z$ is continuous into $X$ by the coreflection property of $k_\Delta$. Since $\I$ is compact and $X$ is Hausdorff, this makes $J_b$ an embedded interval with its original topology. Moreover,
\begin{equation}\label{ce:eq:intersection}
J_b\cap J_{b'}=\{u,v\}\qquad(b\ne b').
\end{equation}
The open subset $F\setminus(E_-\cup E_+)$ is saturated for $\rho$, and its equivalence classes are singletons. The restriction of $\rho$ to this subset is consequently a homeomorphism onto $Z\setminus\{u,v\}$. Projection to $B$ gives the continuous index map
\begin{equation}\label{ce:eq:interior-index}
\beta:Z\setminus\{u,v\}\longrightarrow B.
\end{equation}

\subsection{Saturation and countable restrictions}
\label{ce:sec:saturation}

Give each $J_b$ the order transported from $F_b$. Declare a continuous path in $X$ directed if it is constant or is a nondecreasing path contained in one $J_b$. Such a path need not traverse all of $J_b$.

\begin{proposition}\label{ce:prop:saturated}
This directed structure is saturated.
\end{proposition}
\begin{proof}
The directed paths are closed under nondecreasing reparametrization. For a nonconstant path in $J_b$, monotonicity gives
\[
x(1)\ne u,\qquad x(0)\ne v.
\]
Together with \eqref{ce:eq:intersection}, this places any two composable nonconstant directed paths in the same interval. Their concatenation is nondecreasing there; constant paths also satisfy the axioms.

Let $\phi\in\Rep$ and suppose $x\circ\phi$ is directed. Surjectivity gives $x(\I)=(x\circ\phi)(\I)$. If this image is not a singleton, it lies in some $J_b$. For $s<t$ and $r_s\in\phi^{-1}(s)$, $r_t\in\phi^{-1}(t)$,
\[
s<t\ \Longrightarrow\ r_s<r_t\ \Longrightarrow\
x(s)=(x\circ\phi)(r_s)\leq_{J_b}(x\circ\phi)(r_t)=x(t).
\]
Thus $x$ is directed.
\end{proof}

Throughout the rest of this section, put $P=\dP(X)(u,v)$, and let $P_{\mathrm{reg}}$ denote its regular paths with the $\Delta$-kelleyfication of the relative compact-open topology. By Lemma~\ref{ce:lem:fibers}, each $F_b$ has a specified order coordinate: $s$ if $b\notin A$ and $t_b$ if $b\in A$. Its inverse, followed by $\rho$, gives an increasing homeomorphism $r_b:\I\to J_b$. Every path $x\in P$ lies in a unique $J_b$ and traverses it completely, because its endpoints are $u$ and $v$. Consequently,
\[
x=r_b\circ\phi,\qquad \phi=r_b^{-1}\circ x\in\Rep.
\]
The paths with image $J_b$ therefore have the same trace. Distinct intervals give distinct traces because surjective reparametrization preserves images, so the trace set is $B$. Denote the index maps by
\begin{equation}\label{ce:eq:Q}
Q:P\longrightarrow B,\qquad Q_{\mathrm{reg}}:P_{\mathrm{reg}}\longrightarrow B
\end{equation}
where both maps send a path in $J_b$ to $b$.

\begin{proposition}\label{ce:prop:countable}
If $C\subset B$ is compact and $C\cap A$ is countable, there is a homeomorphism
\[
\Theta_C:F_C:=p^{-1}(C)\xrightarrow{\cong}C\times \I
\]
over $C$, strictly increasing on every fiber and preserving both endpoints.
\end{proposition}
\begin{proof}
Choose a summable family $(w_a)_{a\in C\cap A}$ of positive weights and put $W=\sum_{a\in C\cap A}w_a$, with $W=0$ when $C\cap A=\varnothing$. Define
\begin{equation}\label{ce:eq:weighted-height}
h_C=\frac{s+\sum_{a\in C\cap A}w_at_a}{1+W}:F_C\longrightarrow\I.
\end{equation}
The bound $\sup_{F_C}|w_at_a|\leq w_a$ gives uniform convergence and continuity of $h_C$ \cite[Theorem~21.6]{zbMATH01461253}.

For $x<y$ in $F_b$, Lemma~\ref{ce:lem:fibers} gives
\[
\begin{aligned}
(1+W)(h_C(y)-h_C(x))
&=s(y)-s(x)+\sum_{a\in C\cap A}w_a(t_a(y)-t_a(x))\\
&\geq\begin{cases}
s(y)-s(x)>0,&s(x)<s(y),\\
w_b(t_b(y)-t_b(x))>0,&s(x)=s(y).
\end{cases}
\end{aligned}
\]
In the second case, $b\in C\cap A$ and $x,y$ lie in its inserted interval. Since $h_C$ takes the values $0,1$ at the endpoints, it maps each fiber bijectively onto $\I$. Compactness now makes $\Theta_C=(p,h_C)$ a homeomorphism \cite[Theorem~26.6]{zbMATH01461253}:
\[
\begin{tikzcd}[column sep=large]
 F_C\arrow[rr,"\Theta_C","\cong"']\arrow[dr,"p"']&&C\times\I\arrow[dl,"\operatorname{pr}_C"]\\
 &C.&
\end{tikzcd}
\]
\end{proof}

\subsection{The trace topology}
\label{ce:sec:trace-topology}

\begin{lemma}\label{ce:lem:Q-continuous}
The index map $Q:P\to B$ is continuous.
\end{lemma}
\begin{proof}
The open sets $V_t=\{x\in P:x(t)\notin\{u,v\}\}$ cover $P$, and on each of them
\[
Q|_{V_t}:V_t\xrightarrow{\operatorname{ev}_t}X\setminus\{u,v\}\longrightarrow Z\setminus\{u,v\}\xrightarrow{\beta}B
\]
is continuous.
\end{proof}

\begin{lemma}\label{ce:lem:curve-lifting}
If $\alpha:\I\to B$ is continuous and $\alpha(\I)\cap A$ is countable, there is a continuous lift $\widetilde\alpha:\I\to P$ of $\alpha$ whose values are regular paths.
\end{lemma}
\begin{proof}
For $C=\alpha(\I)$, Proposition~\ref{ce:prop:countable} gives the continuous adjoint
\[
e:\I\times\I\longrightarrow Z,\qquad
 e(r,t)=\rho\bigl(\Theta_C^{-1}(\alpha(r),t)\bigr).
\]
Each $e(r,-)$ is an increasing homeomorphism onto $J_{\alpha(r)}$. Lemma~\ref{lem:path-families} therefore gives a continuous lift into the regular path space:
\[
\begin{tikzcd}[column sep=large,row sep=large]
 &P_{\mathrm{reg}}\arrow[r,hook]\arrow[d,"Q_{\mathrm{reg}}"]&P\arrow[dl,"Q"]\\
 \I\arrow[ur,dashed,"\widetilde\alpha"]\arrow[r,"\alpha"']&B.&
\end{tikzcd}
\]
\end{proof}

\begin{lemma}\label{ce:lem:quotient-curves}
Let $D$ be a convex subset of a normed space and $\pi:E\to D$ a continuous surjection. Suppose every continuous curve $\alpha:\I\to D$ whose image is contained in a countable union of straight segments and one point has a continuous lift to $E$. Then $\pi$ is a quotient map.
\end{lemma}
\begin{proof}
If $U\subset D$ is not open, the sequential criterion \cite[Lemma~21.2]{zbMATH01461253} gives $b\in U$ and $b_n\in D\setminus U$ with $b_n\to b$. Choose $1=t_1>t_2>\cdots\to0$ and interpolate:
\[
\begin{gathered}
\alpha(0)=b,\qquad
\theta_n(t)=\frac{t-t_{n+1}}{t_n-t_{n+1}}\in\I\quad(t\in[t_{n+1},t_n]),\\
\alpha(t)=(1-\theta_n(t))b_{n+1}+\theta_n(t)b_n,\\
\begin{aligned}
\|\alpha(t)-b\|
&\leq(1-\theta_n(t))\|b_{n+1}-b\|+\theta_n(t)\|b_n-b\|\\
&\leq\max\{\|b_{n+1}-b\|,\|b_n-b\|\}\longrightarrow0.
\end{aligned}
\end{gathered}
\]
This is a continuous curve in $D$ of the required form, so there is a continuous lift
\[
\begin{tikzcd}[column sep=large,row sep=large]
 &E\arrow[d,"\pi"]\\
 \I\arrow[ur,dashed,"\widetilde\alpha"]\arrow[r,"\alpha"']&D.
\end{tikzcd}
\]
But
\[
\begin{gathered}
\pi^{-1}(U)\text{ open}\ \Longrightarrow\
\alpha^{-1}(U)=\widetilde\alpha^{-1}(\pi^{-1}(U))\text{ open},\\
0\in\alpha^{-1}(U),\qquad t_n\notin\alpha^{-1}(U).
\end{gathered}
\]
This is impossible because $t_n\to0$.
\end{proof}

\begin{proposition}\label{ce:prop:trace-square}
The index map $Q:P\to B$ is a quotient map for the usual topology of $B$. Consequently, $\dT(X)(u,v)\cong B$ in $\TopD$ and in $\TopDH$.
\end{proposition}
\begin{proof}
The map $Q$ is continuous by Lemma~\ref{ce:lem:Q-continuous} and surjective by Lemma~\ref{ce:lem:fibers}. A countable union of straight segments and one point meets $A$ countably, so Lemmas~\ref{ce:lem:curve-lifting} and~\ref{ce:lem:quotient-curves} make $Q$ a quotient map. Its fibers are the trace classes, and $B$ is Hausdorff and $\Delta$-generated.
\end{proof}

\subsection{The boundary loop and regular directed paths}
\label{ce:sec:boundary}

\begin{lemma}\label{ce:lem:countability-obstruction}
There is no continuous surjection from a second countable space onto $Z$.
\end{lemma}
\begin{proof}
The closed middle slice has relatively open subsets
\begin{equation}\label{ce:eq:closed-slice}
M=\{s=1/2\}\subset Z,\qquad U_a=M\cap\{t_a>1/2\}\quad(a\in A).
\end{equation}
Since $g_a(b,1/2)=1/2$, strict monotonicity for $b\ne a$ gives
\[
\begin{gathered}
b\ne a\ \Longrightarrow\ t_a|_{M\cap J_b}=1/2,\qquad
U_a\cong]1/2,2/3]\subset J_a,\\
U_a\ne\varnothing,\qquad U_a\cap U_{a'}=\varnothing\quad(a\ne a').
\end{gathered}
\]
A continuous surjection $e:K\to Z$ would give uncountably many pairwise disjoint nonempty open subsets $e^{-1}(U_a)$ of $e^{-1}(M)$. A second countable space, and hence each of its subspaces, has at most countably many such subsets: each contains a distinct element of a countable basis.
\end{proof}

For $b\in\partial B$ and $a\in A$,
\[
\|b-a\|\geq\|b\|-\|a\|\geq2-1=1
\quad\Longrightarrow\quad\delta_a(b)=1,\quad g_a(b,t)=t.
\]
Thus the boundary admits the family
\begin{equation}\label{ce:eq:boundary-loop}
\ell:\partial B\longrightarrow P,
\qquad
\ell(b)(t)=\rho\bigl(b,t,(t)_{a\in A}\bigr).
\end{equation}
Its adjoint is continuous into $Z$, so Lemma~\ref{lem:path-families} makes $\ell$ continuous. Each path traverses $J_b$ strictly increasingly, and
\begin{equation}\label{ce:eq:boundary-index}
Q\ell=\operatorname{incl}_{\partial B\subset B}.
\end{equation}
Choose $b_*=(2,0)\in\partial B$ and the basepoint $\ell(b_*)$ of $P$.

\begin{lemma}\label{ce:lem:boundary-surjection}
A continuous map $f:B\to B$ with $f|_{\partial B}=\id$ is surjective.
\end{lemma}
\begin{proof}
If $z\in B\setminus f(B)$, then $z$ is interior. Radial projection $r_z:B\setminus\{z\}\to\partial B$ would give the extension
\[
\begin{tikzcd}[column sep=large,row sep=large]
\partial B\arrow[r,equal]\arrow[d,hook]&\partial B\arrow[d]\\
B\arrow[r]\arrow[ur,dashed,"r_zf"]&\mathbf1,
\end{tikzcd}
\]
a retraction forbidden by \cite[Corollary~2.15]{HatcherAlgebraicTopology-full}.
\end{proof}

\begin{proof}[Proof of Theorem~\ref{ce:thm:main}]
The structural assertions follow from Proposition~\ref{ce:prop:saturated} and Proposition~\ref{ce:prop:trace-square}. Suppose $\ell$ were nullhomotopic. Its extension $L$ would give
\[
\begin{tikzcd}[column sep=large,row sep=large]
 \partial B\arrow[r,"\ell"]\arrow[d,hook]&P\arrow[d,"Q"]\\
 B\arrow[r,"f=QL"']\arrow[ur,"L",dashed]&B.
\end{tikzcd}
\]
By \eqref{ce:eq:boundary-index} and Lemma~\ref{ce:lem:boundary-surjection}, $f(B)=B$. Lemma~\ref{lem:path-families} makes the evaluation
\begin{equation}\label{ce:eq:surjective-evaluation}
e_L:B\times\I\longrightarrow Z,\qquad e_L(b,t)=L(b)(t)
\end{equation}
continuous, and
\[
e_L(B\times\I)=\bigcup_{b\in B}L(b)(\I)
=\bigcup_{b\in B}J_{f(b)}=\bigcup_{c\in B}J_c=Z.
\]
This contradicts Lemma~\ref{ce:lem:countability-obstruction}, since $B\times\I$ is second countable. Hence $[\ell]\ne0$ in $\pi_1(P,\ell(b_*))$, whereas $Q_*[\ell]=0$ because $B$ is contractible.
\end{proof}

\begin{remark}\label{ce:rem:ordinary-topology}
Write $P_{\co}=\dP_{\co}(X)(u,v)$ for the ordinary compact-open space and $P_{\mathrm{reg},\co}$ for its regular subspace. Disks, spheres, and their products with $\I$ are $\Delta$-generated, so the coreflection property gives
\[
\pi_n(k_\Delta U,u)\xrightarrow{\cong}\pi_n(U,u)\quad(n\geq0),\qquad
P=k_\Delta P_{\co},\quad P_{\mathrm{reg}}=k_\Delta P_{\mathrm{reg},\co}.
\]
Thus Theorem~\ref{ce:thm:main} also holds for $P_{\co}$. The evaluation formula of Lemma~\ref{ce:lem:Q-continuous} defines continuous index maps $Q_{\co}:P_{\co}\to B$ and $Q_{\mathrm{reg},\co}:P_{\mathrm{reg},\co}\to B$. The first is quotient because its composite with $P\to P_{\co}$ is $Q$.
\end{remark}

\begin{lemma}\label{ce:lem:quotient-fibers}
For every $b\in B$, the map $\phi\mapsto r_b\circ\phi$ is a homeomorphism from $\Rep$ onto the fiber over $b$, both for $Q$ and for $Q_{\co}$. Its restriction identifies $\operatorname{Homeo}_+(\I)$ with the regular paths in that fiber, in either topology. Both spaces are contractible.
\end{lemma}
\begin{proof}
The embedded interval $J_b$ gives inverse continuous maps on compact-open fibers by composition with $r_b$ and $r_b^{-1}$. Regularity is equivalent to strict increase of the reparametrization, so
\[
Q_{\co}^{-1}(b)\cong\Rep,\qquad
Q_{\mathrm{reg},\co}^{-1}(b)\cong\operatorname{Homeo}_+(\I).
\]
Both domains are metrizable and locally path-connected, hence $\Delta$-generated; the coreflection property gives the same homeomorphisms for the relative fibers of $Q$ and $Q_{\mathrm{reg}}$. Both spaces are convex \cite[Lemma~2.6\textup{(1)}]{reparam}, with contractions
\[
H(\phi,s)=(1-s)\phi+s\id_\I.
\]
\end{proof}

\begin{corollary}\label{ce:cor:not-quasifibration}
The quotient map $Q:P\to B$ is not a quasifibration. The same holds for $Q_{\co}:P_{\co}\to B$.
\end{corollary}
\begin{proof}
By Lemma~\ref{ce:lem:quotient-fibers}, every fiber is contractible. A quasifibration has the homotopy exact sequence of a fibration; see \cite[Section~7]{May}. Put $p_* = \ell(b_*)$. If $Q$ were a quasifibration, this sequence would contain
\[
\begin{tikzcd}[column sep=large]
0=\pi_1(Q^{-1}(b_*),p_*)\arrow[r]&\pi_1(P,p_*)\arrow[r,"Q_*"]&\pi_1(B,b_*)=0.
\end{tikzcd}
\]
Exactness would force $\pi_1(P,p_*)=0$, contrary to Theorem~\ref{ce:thm:main}. The same argument applies to $Q_{\co}$ by Remark~\ref{ce:rem:ordinary-topology}.
\end{proof}

\begin{remark}\label{ce:rem:no-clock}
Theorem~\ref{thm:timed-traces} and Theorem~\ref{ce:thm:main} give
\[
X\xrightarrow{\text{regular}}C\text{ with $C$ a clock}
\ \Longrightarrow\ Q\text{ a weak homotopy equivalence}\ \Longrightarrow\ \bot.
\]
Thus $X$ admits no regular directed map to a clock. In particular, $|X|$ is not submetrizable (otherwise $\id_X$ would be such a map), and hence is not metrizable. By Theorem~\ref{thm:metrizable-quasicompacts-traces}, it has a compact subspace which is not metrizable.
\end{remark}

\begin{corollary}\label{ce:cor:regular}
The quotient of $P_{\mathrm{reg}}$ by the group $\operatorname{Homeo}_+(\I)$ of increasing homeomorphisms is homeomorphic to $B$, and
\[
Q_{\mathrm{reg}}:P_{\mathrm{reg}}\longrightarrow P_{\mathrm{reg}}/\operatorname{Homeo}_+(\I)\cong B
\]
is neither a weak homotopy equivalence nor a quasifibration. The conclusions also hold for the ordinary compact-open topology and ordinary topological quotient.
\end{corollary}
\begin{proof}
The regular paths in $J_b$ are exactly its increasing parametrizations; thus the orbits are the fibers of $Q_{\mathrm{reg}}$. Lemma~\ref{ce:lem:curve-lifting} solves
\[
\begin{tikzcd}[column sep=large,row sep=large]
 &P_{\mathrm{reg}}\arrow[d,"Q_{\mathrm{reg}}"]\\
 \I\arrow[ur,dashed,"\widetilde\alpha"]\arrow[r,"\alpha"']&B
\end{tikzcd}
\]
whenever $\alpha(\I)\cap A$ is countable. Lemma~\ref{ce:lem:quotient-curves} therefore makes $Q_{\mathrm{reg}}$ a quotient map.

The family $\ell$ takes values continuously in $P_{\mathrm{reg}}$ by Lemma~\ref{lem:path-families}. Let $\iota:P_{\mathrm{reg}}\to P$ be the inclusion. A nullhomotopy would give
\[
\begin{tikzcd}[column sep=large,row sep=large]
\partial B\arrow[r,"\ell"]\arrow[d,hook]&P_{\mathrm{reg}}\arrow[d]\\
B\arrow[r]\arrow[ur,dashed,"L_{\mathrm{reg}}"]&\mathbf1
\end{tikzcd}
\quad\Longrightarrow\quad
\begin{tikzcd}[column sep=large,row sep=large]
\partial B\arrow[r,"\ell"]\arrow[d,hook]&P\arrow[d]\\
B\arrow[r]\arrow[ur,dashed,"\iota L_{\mathrm{reg}}"]&\mathbf1,
\end{tikzcd}
\]
contradicting Theorem~\ref{ce:thm:main}. Hence $Q_{\mathrm{reg}}$ kills a nonzero fundamental group element. Its contractible fibers (Lemma~\ref{ce:lem:quotient-fibers}) also exclude a quasifibration by the exact sequence of Corollary~\ref{ce:cor:not-quasifibration}. The coreflection argument of Remark~\ref{ce:rem:ordinary-topology} gives the same conclusions for the ordinary compact-open topology, whose index map is quotient because its composite with coreflection is $Q_{\mathrm{reg}}$.
\end{proof}

The surjectivity on homotopy groups asserted for $Q$ in \cite[Remark~4.6]{reparam} is compatible with this example, whose trace space is contractible. Corollary~\ref{ce:cor:regular} contradicts the weak homotopy equivalence claimed for $Q_R$ in the same diagram for regular directed paths. The proof for ordinary regular paths in \cite[Corollary~3.5]{reparam} infers a fibration from a free reparametrization action; freeness does not imply that the orbit map is a fibration. The homeomorphism between the two trace spaces in \cite[Corollary~4.5]{reparam} remains valid.

\section{The category over a fixed clock}
\label{sec:timed-category}

\begin{definition}
A \emph{clock} is a directed space $C$ for which $|C|$ is \emph{submetrizable}: there is a continuous injective map into a metric space, or equivalently a metrizable topology on the same set coarser than its given topology. The metric and the injection are not part of the data. A clock is not required to be saturated.

A \emph{timed space over $C$} is a regular directed map $p:X\to C$ with $X$ saturated. A morphism from $p:X\to C$ to $q:Y\to C$ is a directed map $h:X\to Y$ such that $qh=p$. These objects and morphisms form the category $\Timed(C)$.
\[
\begin{tikzcd}[column sep=large]
 X\arrow[rr,"h"]\arrow[dr,"p"']&&Y\arrow[dl,"q"]\\
 &C&
\end{tikzcd}
\]
\end{definition}

Local presentability does not require the underlying space of $C$ to be submetrizable. Proposition~\ref{prop:saturated-clock-part} reduces the proof to the case of a saturated clock.

For a continuous path $z$ and $0\leq a\leq b\leq1$, write
\[
z_{[a,b]}(t)=z(a+(b-a)t).
\]
This includes the constant path when $a=b$.

\begin{lemma}\label{lem:trace-concatenation}
Trace equivalence of continuous paths is compatible with concatenation.
\end{lemma}
\begin{proof}
For $\phi,\psi\in\Rep$, the map
\[
\begin{gathered}
\Phi(t)=
\begin{cases}
\frac12\phi(2t),&0\leq t\leq\frac12,\\
\frac12+\frac12\psi(2t-1),&\frac12\leq t\leq1
\end{cases}
\qquad\Phi\in\Rep,\\
(a\circ\phi)*_N(b\circ\psi)=(a*_N b)\circ\Phi.
\end{gathered}
\]
The generating identifications therefore preserve concatenation.
\end{proof}

\begin{proposition}\label{lem:trace-splitting}
Trace equivalence transports finite subdivisions: if $z\sim c$ and $0=t_0\leq\cdots\leq t_m=1$, there are $0=r_0\leq\cdots\leq r_m=1$ such that
\[
z_{[r_{i-1},r_i]}\sim c_{[t_{i-1},t_i]}\qquad(1\leq i\leq m).
\]
In particular, $z\sim a*_N b$ implies $z_{[0,r]}\sim a$ and $z_{[r,1]}\sim b$ for some $r\in\I$.
\end{proposition}
\begin{proof}
For a generating identification $c\circ\phi\sim c$, put $q=\phi(u)$ and $q'=\phi(v)$, where $u\leq v$. Restriction gives
\[
\begin{tikzcd}[column sep=large,row sep=large]
{[u,v]}\arrow[r,hook]\arrow[d,two heads,"\phi|_{[u,v]}"']&\I\arrow[d,two heads,"\phi"]\\
{[q,q']}\arrow[r,hook]&\I.
\end{tikzcd}
\]
If $q<q'$, the normalized restriction $\phi_{u,v}\in\Rep$ satisfies
\[
\phi_{u,v}(t)=\frac{\phi(u+(v-u)t)-q}{q'-q},
\qquad (c\circ\phi)_{[u,v]}=c_{[q,q']}\circ\phi_{u,v}.
\]
If $q=q'$, both paths are constant. Thus every subdivision of $c\circ\phi$ descends under $\phi$. Conversely, choose ordered $r_i\in\phi^{-1}(t_i)$, with $r_0=0$ and $r_m=1$, to transport a subdivision of $c$ to $c\circ\phi$. Applying these two operations along a finite chain of generating identifications proves the claim.
\end{proof}

\begin{lemma}\label{lem:regular-reparametrization-descent}
Let $p:X\to C$ be a regular directed map with $X$ saturated. If $x$ is directed, $y:\I\to|C|$ is continuous, and $p\circ x=y\circ\phi$ for $\phi\in\Rep$, then there is a unique directed path $x'$ in the diagram
\[
\begin{tikzcd}[column sep=large,row sep=large]
\I\arrow[r,"x"]\arrow[d,two heads,"\phi"']&{|X|}\arrow[d,"p"]\\
\I\arrow[r,"y"']\arrow[ur,dashed,"x'"]&{|C|}.
\end{tikzcd}
\]
\end{lemma}
\begin{proof}
The fibers of $\phi$ are closed intervals, and Lemma~\ref{lem:regular-stops} gives
\[
(p\circ x)|_{\phi^{-1}(t)}\text{ constant}
\ \Longrightarrow\ x|_{\phi^{-1}(t)}\text{ constant}\qquad(t\in\I).
\]
The quotient argument of \cite[Lemma~5.3]{clockmap} gives the unique continuous factor $x'$. Saturation and surjectivity yield
\[
x=x'\circ\phi\in dX\ \Longrightarrow\ x'\in dX,
\qquad p\circ x'\circ\phi=y\circ\phi\ \Longrightarrow\ p\circ x'=y.
\]
\end{proof}

\begin{proposition}\label{prop:saturation-structures}
For a directed space $C$, both collections
\[
\begin{aligned}
d(C^{\mathrm{sat}})&=\{c\in\Top(\I,|C|):c\sim a\text{ for some }a\in dC\},\\
d(C^{\circ})&=\{c\in\Top(\I,|C|):\forall a\in\Top(\I,|C|),\ a\sim c\Longrightarrow a\in dC\}
\end{aligned}
\]
define saturated directed structures on $|C|$, with
\[
d(C^{\circ})\subseteq dC\subseteq d(C^{\mathrm{sat}}).
\]
Both inclusions are equalities if $C$ is saturated.
\end{proposition}
\begin{proof}
Both collections contain constants and whole trace classes. For $C^{\mathrm{sat}}$, Proposition~\ref{lem:trace-splitting} and Lemma~\ref{lem:trace-concatenation} give
\[
\begin{gathered}
c\sim a\in dC\ \Longrightarrow\
c_{[u,v]}\sim a_{[r,s]}\in dC\text{ for some }r\leq s,\\
c_1\sim a_1\in dC,\quad c_2\sim a_2\in dC
\ \Longrightarrow\ c_1*_N c_2\sim a_1*_N a_2\in dC.
\end{gathered}
\]
For $c\in d(C^{\circ})$ and $a<b$, compatibility with concatenation gives
\[
v\sim c_{[a,b]}
\ \Longrightarrow\ c_{[0,a]}*_N(v*_N c_{[b,1]})\sim c
\ \Longrightarrow\ v\in dC.
\]
Hence $c_{[a,b]}\in d(C^{\circ})$; degenerate restrictions are constant. For concatenation, Proposition~\ref{lem:trace-splitting} gives
\[
\begin{aligned}
z\sim c_1*_N c_2,\quad c_1,c_2\in d(C^{\circ})
&\Longrightarrow\ z_{[0,r]}\sim c_1,\quad z_{[r,1]}\sim c_2\quad\text{for some }r\\
&\Longrightarrow\ z_{[0,r]},z_{[r,1]}\in dC\ \Longrightarrow\ z\in dC.
\end{aligned}
\]
The last implication uses concatenation and an increasing piecewise linear reparametrization when $0<r<1$; for $r=0,1$, one piece equals $z$. Finally, every nonconstant nondecreasing map $\psi:\I\to\I$ factors as
\[
\psi(t)=a+(b-a)\bar\psi(t),\qquad a=\psi(0)<\psi(1)=b,\quad\bar\psi\in\Rep.
\]
Closure under restrictions and $\Rep$ thus gives all nondecreasing reparametrizations. The inclusions and the equality assertion follow from the definitions.
\end{proof}

\begin{proposition}\label{prop:saturated-clock-part}
The identity on underlying spaces defines a regular directed map $\iota_C:C^{\circ}\to C$. Every regular directed map $p:X\to C$ with $X$ saturated factors uniquely through $\iota_C$ as a regular directed map $p^{\circ}:X\to C^{\circ}$:
\[
\begin{tikzcd}[column sep=large]
 X\arrow[rr,dashed,"p^{\circ}"]\arrow[dr,"p"']&&C^{\circ}\arrow[dl,"\iota_C"]\\
 &C.&
\end{tikzcd}
\]
\end{proposition}
\begin{proof}
The map $\iota_C$ is directed and injective, hence regular. A generating trace identification can be followed by composition, $(p\circ x)\circ\phi=p\circ(x\circ\phi)$, or reversed by Lemma~\ref{lem:regular-reparametrization-descent}:
\[
\begin{tikzcd}[column sep=large,row sep=large]
\I\arrow[r,"x"]\arrow[d,two heads,"\phi"']&{|X|}\arrow[d,"p"]\\
\I\arrow[r,"y"']\arrow[ur,dashed,"x'"]&{|C|}.
\end{tikzcd}
\]
A finite chain therefore gives
\[
x\in dX,\quad y\sim p\circ x\ \Longrightarrow\ y\in p(dX)\subseteq dC
\quad\Longrightarrow\quad p\circ x\in d(C^{\circ}).
\]
The factorization has the same underlying map as $p$, so it is unique and regular.
\end{proof}

\begin{proposition}\label{rem:saturation-adjunctions}
Saturation is left adjoint to the full inclusion $\sdTop\hookrightarrow\dTop$. Let $\dTop_{\mathrm{reg}}$ have the same objects as $\dTop$ and only regular directed maps, and let $J:\sdTop_{\mathrm{reg}}\hookrightarrow\dTop_{\mathrm{reg}}$ be the full inclusion of saturated objects. Then
\[
(-)^{\mathrm{sat}}\dashv J\dashv(-)^{\circ},\qquad
\begin{tikzcd}[column sep=large]
C^{\circ}\arrow[r,"\iota_C"]&C\arrow[r,"\eta_C"]&C^{\mathrm{sat}}.
\end{tikzcd}
\]
Both adjoints preserve underlying spaces, and hence preserve the clock condition. Thus $\sdTop_{\mathrm{reg}}$ is reflective and coreflective in $\dTop_{\mathrm{reg}}$; $\dTop_{\mathrm{reg}}$ is not full in $\dTop$.
\end{proposition}
\begin{proof}
The structure $d(C^{\mathrm{sat}})$ is the smallest saturated one containing $dC$; compare \cite[Remark~4.3\textup{(i)}]{FR-full}. For saturated $Y$, postcomposition preserves trace equivalence, giving the unique extension
\[
\begin{tikzcd}[column sep=large,row sep=large]
 C\arrow[r,"f"]\arrow[d,"\eta_C"']&Y\arrow[d]\\
 C^{\mathrm{sat}}\arrow[r]\arrow[ur,dashed,"\overline f"]&\mathbf1.
\end{tikzcd}
\]
If $f$ is regular, then for $c\sim a\in dC$,
\[
\overline f\circ c\text{ constant}
\ \Longrightarrow\ f\circ a\text{ constant}
\ \Longrightarrow\ a\text{ constant}
\ \Longrightarrow\ c\text{ constant},
\]
since trace equivalence preserves images. Hence $\overline f$ is regular as well. Proposition~\ref{prop:saturated-clock-part} gives the right adjoint. Both constructions change only the directed paths.
\end{proof}

\begin{theorem}\label{thm:timed-locally-presentable}
For every clock $C$, the category $\Timed(C)$ is locally presentable. Every morphism in this category is regular.
\end{theorem}
\begin{proof}
Proposition~\ref{prop:saturated-clock-part} gives $\Timed(C)\cong\Timed(C^{\circ})$, so assume $C$ saturated. The category $\mathcal A=\sdTop/C$ is locally presentable by \cite[Theorem~2.4]{clockmap} and the slice argument of \cite[Proposition~3.3]{clockmap}. We apply the orthogonality argument of \cite[Proposition~3.4 and Corollary~3.5]{clockmap}.

For $c\in|C|$, regard the constant maps $\vec\I_c\to C$ and $\mathbf1_c\to C$ as objects of $\mathcal A$, and let $\rho_c:\vec\I_c\to\mathbf1_c$ be the unique map. For $p:X\to C$,
\[
\begin{aligned}
\mathcal A(\mathbf1_c,p)&\cong p^{-1}(c),\\
\mathcal A(\vec\I_c,p)&\cong\{x\in dX:p\circ x\text{ is constant at }c\},\\
\rho_c^*:p^{-1}(c)&\longrightarrow\mathcal A(\vec\I_c,p),
\qquad u\longmapsto\operatorname{const}_u.
\end{aligned}
\]
Consequently, $p$ is regular if and only if every square below has a unique diagonal:
\[
\begin{tikzcd}[column sep=large,row sep=large]
 \vec\I_c\arrow[r,"x"]\arrow[d,"\rho_c"']&X\arrow[d,"p"]\\
 \mathbf1_c\arrow[r,"c"']\arrow[ur,dashed,"\exists!"]&C.
\end{tikzcd}
\]
In the slice, this says
\[
\Timed(C)=\{\rho_c:c\in|C|\}^{\perp}\subseteq\mathcal A.
\]
This orthogonality class is defined by a set of morphisms, so it is locally presentable by \cite[Corollary~1.40]{TheBook}.

For a morphism $f:(X,p)\to(Y,q)$ and $x\in dX$,
\[
f\circ x\text{ constant}\ \Longrightarrow\ p\circ x=q\circ f\circ x\text{ constant}
\ \Longrightarrow\ x\text{ constant}.
\]
Thus $f$ is regular.
\end{proof}
\section{A progress function on metric spaces}
\label{sec:metric}

\begin{definition}[Progress function]
\label{def:progress-function}
Let $Z$ be a topological space and let $\mathcal A\subseteq\Cont_\co(\I,Z)$ be a family of paths stable under $\Rep$, with its relative compact-open topology. A \emph{progress function} on $\mathcal A$ is a continuous map
\[
L:\mathcal A\times\I\longrightarrow[0,\infty[,
\qquad (x,t)\longmapsto L_x(t),
\]
such that each $L_x$ is nondecreasing, $L_x(0)=0$, and
\begin{align}
L_{x\circ\phi}(t)&=L_x(\phi(t)),\label{eq:clock-equivariance}\\
L_x(a)=L_x(b)&\quad\Longleftrightarrow\quad x|_{[a,b]}\text{ is constant}.\label{eq:clock-stop-detection}
\end{align}
Here $x\in\mathcal A$, $\phi\in\Rep$, $t\in\I$, and $0\leq a\leq b\leq1$. We also speak of a progress function on $Z$ when $\mathcal A=\Cont_\co(\I,Z)$, and on a directed space $X$ when $\mathcal A=dX$.
\end{definition}

Let $(Y,d)$ be a metric space. We seek a continuous nondecreasing function with the same constant subintervals as a path $x:\I\to Y$, as in parametrization by arc length. The total variation
\[
\sup_{0=t_0<\cdots<t_r=t}
\sum_{i=1}^{r}d(x(t_{i-1}),x(t_i))
\]
can be infinite. Even paths of finite length can converge uniformly to a constant path without their lengths converging to zero; see \cite[Remark~2.8]{MR2521708}. Thus arc length does not give a continuous function on the space of all paths.

We use the \emph{truncated variation} of {\L}ochowski \cite[p.~122 and Section~2]{LochowskiTruncatedVariation-full}. For $\epsilon>0$ and $t\in\I$, put
\[
\operatorname{Var}_{\epsilon}(x,t)=
\sup_{0=t_0<\cdots<t_r=t}
\sum_{i=1}^{r}
\bigl(d(x(t_{i-1}),x(t_i))-\epsilon\bigr)_{+},
\qquad a_{+}=\max(a,0).
\]
By convention, $\operatorname{Var}_{\epsilon}(x,0)=0$. A distance $r\leq\epsilon$ contributes zero, and a distance $r>\epsilon$ contributes $r-\epsilon$ (Figure~\ref{fig:truncated-variation}). This is $TV^\epsilon(x,[0,t])$ in \cite[p.~122, definition following (1.1)]{LochowskiTruncatedVariation-full}: adding the endpoints to the partitions used there adds nonnegative summands and gives the same supremum.

\begin{figure}[!htbp]
\centering
\begin{tikzpicture}[x=1cm,y=1cm,>=stealth,font=\small]
 \begin{scope}
  \node at (2.1,3.35) {One partition interval};
  \draw[->] (0,0)--(4.5,0) node[right] {$t$};
  \draw[->] (0,0)--(0,2.85) node[left] {$x(t)$};
  \draw[thick] (.2,.55)--(.45,.8)--(.7,.55)--(.95,.8)--(1.2,.55)--(1.5,.8)--(2.65,2.4)--(3.6,2.4);
  \draw[densely dashed,gray] (1.5,0)--(1.5,.8)--(3.95,.8);
  \draw[densely dashed,gray] (2.65,0)--(2.65,2.4)--(3.95,2.4);
  \draw[densely dashed,gray] (3.65,1.4)--(3.95,1.4);
  \fill (1.5,.8) circle (1.5pt);
  \fill (2.65,2.4) circle (1.5pt);
  \node[below] at (1.5,0) {$t_{i-1}$};
  \node[below] at (2.65,0) {$t_i$};
  \draw[<->,gray] (3.8,.8)--(3.8,1.4) node[midway,right] {$\epsilon$};
  \draw[<->,thick,dark-red] (3.8,1.4)--(3.8,2.4) node[midway,right] {$r-\epsilon$};
  \node at (2.1,-.8) {$r=|x(t_i)-x(t_{i-1})|>\epsilon$};
 \end{scope}
 \begin{scope}[xshift=7cm]
  \node at (1.85,3.35) {The contribution to the sum};
  \draw[->] (0,0)--(4.15,0) node[right] {$r$};
  \draw[->] (0,0)--(0,2.85);
  \draw[dashed,gray] (0,0)--(2.65,2.65) node[above right] {$r$};
  \draw[very thick,dark-red] (0,0)--(1.05,0)--(3.6,2.55) node[right] {$(r-\epsilon)_+$};
  \draw (1.05,.07)--(1.05,-.07) node[below] {$\epsilon$};
  \node[below left] at (0,0) {$0$};
  \node at (1.85,-.8) {$(r-\epsilon)_+=\max(r-\epsilon,0)$};
 \end{scope}
\end{tikzpicture}
\caption{A summand of the truncated variation. On the left, $r=|x(t_i)-x(t_{i-1})|>\epsilon$, and the colored segment has length $r-\epsilon$. On the right, the dashed line is $r$ and the colored graph is $(r-\epsilon)_+$.}
\label{fig:truncated-variation}
\end{figure}

Finiteness follows from \cite[Fact~2.2 and Remark~2.3]{LochowskiTruncatedVariation-full}. For continuous paths, uniform continuity gives a direct proof and a local estimate proving continuity in the path variable.

Since $\I$ is compact, the compact-open topology on $\Cont_\co(\I,Y)$ is induced by the uniform metric
\[
d_\infty(y,z)=\sup_{r\in\I}d(y(r),z(r)).
\]

\begin{proposition}
\label{prop:truncated-variation-estimate}
For every $\epsilon>0$, the truncated variation $\operatorname{Var}_\epsilon(x,t)$ is finite. For each path $x_0$, there are an open neighborhood $U$ of $x_0$ and an integer $M\geq1$ such that
\[
\left|\operatorname{Var}_\epsilon(y,t)-\operatorname{Var}_\epsilon(z,t)\right|
\leq2M d_\infty(y,z)
\qquad(y,z\in U,\ t\in\I).
\]
\end{proposition}
\begin{proof}
Choose $0<\delta\leq1$ with $d(x_0(s),x_0(t))<\epsilon/2$ whenever $|s-t|<\delta$. Put $U=B_{d_\infty}(x_0,\epsilon/4)$ and $M=\lfloor\delta^{-1}\rfloor$. Then
\[
y\in U,\quad |s-t|<\delta
\quad\Longrightarrow\quad
d(y(s),y(t))\leq2d_\infty(y,x_0)+d(x_0(s),x_0(t))<\epsilon.
\]
For a partition $\pi=(0=t_0<\cdots<t_r=t)$, write
\[
S_\epsilon(y,\pi)=\sum_{i=1}^r\bigl(d(y(t_{i-1}),y(t_i))-\epsilon\bigr)_+,
\qquad J_\delta(\pi)=\{i\mid t_i-t_{i-1}\geq\delta\}.
\]
At $t=0$, use the partition $(0)$ and the empty sum. Only indices in $J_\delta(\pi)$ contribute, and
\[
\delta\,\#J_\delta(\pi)\leq\sum_{i\in J_\delta(\pi)}(t_i-t_{i-1})\leq t\leq1,
\qquad \#J_\delta(\pi)\leq M.
\]
Since $D=\operatorname{diam}(x_0(\I))+\epsilon/2<\infty$,
\[
0\leq S_\epsilon(y,\pi)\leq MD
\quad\Longrightarrow\quad
\operatorname{Var}_\epsilon(y,t)\leq MD<\infty.
\]
The function $r\mapsto(r-\epsilon)_+$ is $1$-Lipschitz. For $y,z\in U$,
\[
\begin{aligned}
|S_\epsilon(y,\pi)-S_\epsilon(z,\pi)|
&\leq\sum_{i\in J_\delta(\pi)}\left|d(y(t_{i-1}),y(t_i))-d(z(t_{i-1}),z(t_i))\right|\\
&\leq\sum_{i\in J_\delta(\pi)}\bigl(d(y(t_{i-1}),z(t_{i-1}))+d(y(t_i),z(t_i))\bigr)\\
&\leq2M d_\infty(y,z).
\end{aligned}
\]
For every partition $\pi$,
\[
S_\epsilon(y,\pi)\leq S_\epsilon(z,\pi)+2M d_\infty(y,z)\leq\operatorname{Var}_\epsilon(z,t)+2M d_\infty(y,z).
\]
Taking the supremum over $\pi$ and interchanging $y,z$ gives
\[
\left|\operatorname{Var}_\epsilon(y,t)-\operatorname{Var}_\epsilon(z,t)\right|
\leq2M d_\infty(y,z).
\]
\end{proof}

\begin{proposition}
	\label{prop:truncated-variation}
For every $\epsilon>0$, the function
	\[
	\operatorname{Var}_{\epsilon}:
	\Cont_{\co}(\I,Y)\times\I\longrightarrow[0,\infty[
	\]
is finite and continuous.  It has the following properties.
	\begin{enumerate}[label=\textup{(\roman*)}]
\item The function $t\mapsto\operatorname{Var}_{\epsilon}(x,t)$ is nondecreasing.
\item If $0\leq a\leq b\leq1$, then
		\[
		\operatorname{Var}_{\epsilon}(x,a)
		=\operatorname{Var}_{\epsilon}(x,b)
		\]
whenever $x$ is constant on $[a,b]$.  Moreover, if there are $s,t\in[a,b]$ with $s<t$ and $d(x(s),x(t))>\epsilon$, then
		\[
		\operatorname{Var}_{\epsilon}(x,a)
		<\operatorname{Var}_{\epsilon}(x,b).
		\]
\item For every $\phi\in\Rep$ and $t\in\I$,
		\[
		\operatorname{Var}_{\epsilon}(x\circ\phi,t)
		=\operatorname{Var}_{\epsilon}(x,\phi(t)).
		\]
	\end{enumerate}
\end{proposition}

\begin{proof}
Put $R(x,t)(r)=x(tr)$. Since evaluation and multiplication are continuous, so is its adjoint
\[
\widehat R:\Cont_\co(\I,Y)\times\I\times\I\longrightarrow Y,
\qquad (x,t,r)\longmapsto x(tr).
\]
The exponential law \cite[Theorems~46.10--46.11]{zbMATH01461253} gives continuity of $R$, and rescaling gives the commutative diagram
\[
\begin{tikzcd}[column sep=large,row sep=large]
\Cont_\co(\I,Y)\times\I\arrow[r,"R"]\arrow[dr,"\operatorname{Var}_\epsilon"']
&\Cont_\co(\I,Y)\arrow[d,"{\operatorname{Var}_\epsilon(-,1)}"]\\
&{[0,\infty[}.
\end{tikzcd}
\]
Indeed, multiplication by $t>0$ identifies the partition sums, and both variations vanish at $t=0$. Proposition~\ref{prop:truncated-variation-estimate} therefore gives joint continuity.

Extending partitions by endpoints proves \textup{(i)}. If $x|_{[a,b]}$ is constant, replacing every partition point after $a$ by $a$ and deleting repetitions preserves the sum, so
\[
\operatorname{Var}_\epsilon(x,b)\leq\operatorname{Var}_\epsilon(x,a)\leq\operatorname{Var}_\epsilon(x,b).
\]
For $a\leq s<t\leq b$ with $d(x(s),x(t))>\epsilon$, extending a partition of $[0,a]$ by $s,t,b$ gives
\[
\operatorname{Var}_\epsilon(x,b)\geq\operatorname{Var}_\epsilon(x,a)+d(x(s),x(t))-\epsilon>\operatorname{Var}_\epsilon(x,a).
\]
For \textup{(iii)}, use
\[
\begin{tikzcd}[column sep=large]
{[0,t]}\arrow[r,two heads,"\phi"]\arrow[dr,"x\circ\phi"']
&{[0,\phi(t)]}\arrow[d,"x"]\\
&Y.
\end{tikzcd}
\]
If $\phi(t)>0$, applying $\phi$ to a partition and deleting repetitions preserves its sum. Conversely, surjectivity and monotonicity lift every partition $0=u_0<\cdots<u_r=\phi(t)$:
\[
\begin{tikzcd}[column sep=large,row sep=large]
{\{0,\ldots,r\}}\arrow[r,"{i\mapsto u_i}"]\arrow[dr,dashed,"{i\mapsto s_i}"']
&{[0,\phi(t)]}\\
&{[0,t]}\arrow[u,two heads,"\phi"'].
\end{tikzcd}
\qquad 0=s_0<\cdots<s_r=t.
\]
Hence the two collections of partition sums coincide. If $\phi(t)=0$, both variations are zero.
\end{proof}

Truncated variation at one level can vanish on a nonconstant path. A sequence of positive levels tending to zero detects every nonconstant subinterval. Positive summable weights combine these functions without losing continuity.

\begin{notation}
	\label{not:multiscale-profile}
Consider a triple $\mathfrak a=((\epsilon_n)_{n\geq1},(w_n)_{n\geq1},\theta)$ such that
	\begin{enumerate}[label=\textup{(\roman*)}]
\item $\epsilon_n>0$ for every $n$ and $\epsilon_n\to0$;
\item $w_n>0$ for every $n$ and $\sum_{n\geq1}w_n<\infty$;
\item $\theta:[0,\infty[\to[0,1[$ is continuous and strictly increasing, with $\theta(0)=0$.
	\end{enumerate}
For example, the bounded increasing function $\theta(r)=r/(1+r)$ satisfies \textup{(iii)}. The levels and weights remain arbitrary subject to \textup{(i)}--\textup{(ii)}.
\end{notation}

Fix $\mathfrak a$ as in Notation~\ref{not:multiscale-profile}. Define
\begin{equation}
	\label{eq:multiscale-metric-clock}
	L_x(t)=\sum_{n\geq1}w_n\,
	\theta\bigl(\operatorname{Var}_{\epsilon_n}(x,t)\bigr).
\end{equation}

\begin{theorem}[Progress function]
	\label{thm:multiscale-aggregation}
Every metric space admits a progress function. More precisely, for every triple $\mathfrak a$ as in Notation~\ref{not:multiscale-profile}, the series in \eqref{eq:multiscale-metric-clock} defines a progress function on $Y$. In particular, $L_x(1)>0$ if and only if $x$ is nonconstant.
\end{theorem}

\begin{proof}
Proposition~\ref{prop:truncated-variation} and the bound
\[
0\leq\sum_{n=N+1}^{M}w_n\theta\bigl(\operatorname{Var}_{\epsilon_n}(x,t)\bigr)\leq\sum_{n>N}w_n\longrightarrow0
\qquad(M>N)
\]
give a uniformly convergent series of continuous nondecreasing functions. Its sum vanishes at $0$ and is equivariant under $\Rep$. Constant subintervals give equal values term by term. For a nonconstant restriction $x|_{[a,b]}$, choose $a\leq s<t\leq b$ and $n$ with $\epsilon_n<d(x(s),x(t))$. Then
\[
\begin{aligned}
\operatorname{Var}_{\epsilon_n}(x,a)&<\operatorname{Var}_{\epsilon_n}(x,b),\\
L_x(b)-L_x(a)
&\geq w_n\bigl(\theta(\operatorname{Var}_{\epsilon_n}(x,b))-\theta(\operatorname{Var}_{\epsilon_n}(x,a))\bigr)>0.
\end{aligned}
\]
Thus $L$ satisfies \eqref{eq:clock-stop-detection}; its value at $1$ detects nonconstant paths.
\end{proof}

\Needspace{9\baselineskip}
\begin{remark}[Morse's $\mu$-length]
\label{rem:morse-length}
The existence of a progress function on a metric space also follows from Morse's $\mu$-length \cite[pp.~916--918 and Theorem~3]{MorseSpecialParametrization-full}; see \cite[Section~3.2]{PopeApproximation}. For $k\geq2$, Morse uses the functions
\[
m_k(x)=\sup_{0\leq t_1<\cdots<t_k\leq1}\min_{1\leq i<k}d(x(t_i),x(t_{i+1})).
\]
Their weighted sum defines $\mu$-length. Its value on an initial subpath is nondecreasing and has exactly the constant subintervals of that path \cite[properties \textup{(f)}--\textup{(h)}, p.~917]{MorseSpecialParametrization-full}. Theorem~3 of that paper proves continuity in the curve variable for the Fr\'echet distance, hence for uniform convergence; continuity in the endpoint of the initial subpath follows by restriction as in Proposition~\ref{prop:truncated-variation}. Formula~\eqref{eq:multiscale-metric-clock} gives a different construction using bounded functions of truncated variation.
\end{remark}

\section{Normalization of timed paths}
\label{sec:normalization}

Fix a clock $C$, a continuous injective map $\iota:|C|\to(M,d)$ into a metric space, and a triple $\mathfrak a$ as in Notation~\ref{not:multiscale-profile}. Let $p:X\to C$ be a timed space, and write $f=\iota\circ p$.

\subsection{Factoring through the progress function}

For a directed path $x$, put
\[
\ell_x(t)=L_{f\circ x}(t).
\]
By Theorem~\ref{thm:multiscale-aggregation} and Lemma~\ref{lem:regular-stops},
\begin{align}
\ell_{x\circ\phi}(t)&=\ell_x(\phi(t)),\label{eq:timed-equivariance}\\
\ell_x(a)=\ell_x(b)&\quad\Longleftrightarrow\quad
x|_{[a,b]}\text{ is constant}.\label{eq:timed-stops}
\end{align}
Composition with $f$ is continuous for the compact-open topology, so $\ell$ is a progress function on $X$.

For nonconstant $x$, define
\begin{equation}
\label{eq:timed-clock}
\lambda_x(t)=\frac{\ell_x(t)}{\ell_x(1)}.
\end{equation}
This is an element of $\Rep$. Its fibers are precisely the maximal intervals on which $x$ is constant, including intervals reduced to a point. Since $\lambda_x$ is a quotient map, there is a unique continuous path $N_p(x)$ such that
\begin{equation}
\label{eq:timed-factorization}
\begin{tikzcd}[column sep=large]
 \I\arrow[rr,"x"]\arrow[dr,two heads,"\lambda_x"']&&{|X|}\\
 &\I\arrow[ur,dashed,"N_p(x)"']&
\end{tikzcd}
\qquad x=N_p(x)\circ\lambda_x.
\end{equation}
For a constant path set $\lambda_x=\id_\I$ and $N_p(x)=x$. The map $x\mapsto\lambda_x$ need not be continuous at constant paths.

\begin{proposition}
\label{lem:normalization-identities}
For every directed path $x$, the path $N_p(x)$ is directed, has the same endpoints and image as $x$, and represents its trace. Moreover, for every $\phi\in\Rep$,
\[
\begin{aligned}
N_p(x\circ\phi)&=N_p(x),&N_p^2&=N_p,\\
\lambda_{N_p(x)}&=\id_\I,&\ell_{N_p(x)}(t)&=t\ell_x(1)\quad(t\in\I).
\end{aligned}
\]
Every normalized path is regular.
\end{proposition}
\begin{proof}
Saturation and the surjectivity of $\lambda_x$ give directedness, endpoints, image and trace. For nonconstant $x$, equivariance gives
\[
\ell_{N_p(x)}\circ\lambda_x=\ell_x=\ell_x(1)\lambda_x
\quad\Longrightarrow\quad
\ell_{N_p(x)}(t)=t\ell_x(1),\qquad \lambda_{N_p(x)}=\id_\I.
\]
Regularity follows from \eqref{eq:timed-stops}, and factorization with the identity parameter gives $N_p^2=N_p$. For nonconstant $x$ and $\phi\in\Rep$,
\[
\lambda_{x\circ\phi}=\lambda_x\circ\phi,
\qquad
N_p(x)\circ\lambda_x\circ\phi=x\circ\phi=N_p(x\circ\phi)\circ\lambda_x\circ\phi.
\]
Cancellation of the surjection $\lambda_x\circ\phi$ gives $N_p(x\circ\phi)=N_p(x)$. Constant paths satisfy the stated identities by definition.
\end{proof}

For Hausdorff spaces, factorization through a regular path is proved in \cite[Proposition~3.7]{reparam}, with the corrected proof in \cite[Proposition~2.2]{reparam-fixed}. Moreover, \cite[Theorem~3.6 and Corollary~4.5]{reparam} identifies traces with regular paths modulo increasing homeomorphisms. We now prove that the progress function chooses one representative continuously in the original topology.

\subsection{Continuity in the topology of the source}

We give the proof of continuity because a continuous injective map into a metric space need not be an embedding.

\begin{lemma}[Continuous monotone factorization]
\label{lem:factor-continuity}
Let $Z$ be any topological space. Give
\[
E_Z=\{(x,\lambda)\in\Cont_\co(\I,Z)\times\Rep
\mid x\text{ is constant on each fiber of }\lambda\}
\]
the relative ordinary topology. The assignment $(x,\lambda)\mapsto\bar x$, determined by $x=\bar x\circ\lambda$, is continuous into $\Cont_\co(\I,Z)$.
\end{lemma}
\begin{proof}
Every $\lambda\in\Rep$ is a quotient map. For continuity of the adjoint $(x,\lambda,t)\mapsto\bar x(t)$, fix $(x_0,\lambda_0,t_0)$ and an open neighborhood $U$ of $\bar x_0(t_0)$. Choose a closed interval $J=[a,b]$ with
\[
\lambda_0^{-1}(t_0)\subseteq\operatorname{int}_\I J\subseteq J\subseteq x_0^{-1}(U).
\]
The following open neighborhood of $(x_0,\lambda_0,t_0)$ in $E_Z\times\I$ handles interior and endpoint fibers alike:
\[
\mathcal V=\left\{(x,\lambda,t)\ \middle|\ 
\begin{gathered}
x(J)\subseteq U,\\
a=0\text{ or }\lambda(a)<t,\qquad b=1\text{ or }t<\lambda(b)
\end{gathered}
\right\}.
\]
Monotonicity gives
\[
(x,\lambda,t)\in\mathcal V
\quad\Longrightarrow\quad
\lambda^{-1}(t)\subseteq J
\quad\Longrightarrow\quad
\bar x(t)\in U.
\]
The compact-open exponential law finishes the proof.
\end{proof}

\begin{lemma}[Continuity at constant paths]
\label{lem:constant-path-continuity}
Let $E,Z$ be topological spaces and let $F,G:E\to\Cont_\co(\I,Z)$ be maps such that
\[
F(e)(\I)\subseteq G(e)(\I)\qquad(e\in E).
\]
If $G$ is continuous at $e_0$ and $G(e_0)=c_z$ (the constant path $z$), then $F(e_0)=c_z$ and $F$ is continuous at $e_0$.
\end{lemma}
\begin{proof}
The image inclusion gives $F(e_0)=c_z$. A basic neighborhood $\mathcal U=\bigcap_i[K_i,U_i]$ of $c_z$, with empty $K_i$ discarded, contains $[\I,U]$ for $U=\bigcap_iU_i$. Continuity of $G$ supplies a neighborhood $V$ of $e_0$ with
\[
e\in V
\quad\Longrightarrow\quad
F(e)(\I)\subseteq G(e)(\I)\subseteq U
\quad\Longrightarrow\quad
F(e)\in[\I,U]\subseteq\mathcal U.
\]
\end{proof}

\begin{proposition}
\label{prop:normalization-continuous}
For fixed endpoints, the operator
\[
N_p:\dP_\co(X)(u,v)\longrightarrow\dP_\co(X)(u,v)
\]
is continuous. Its $\Delta$-kelleyfication is a continuous operator on $\dP(X)(u,v)$.
\end{proposition}
\begin{proof}
On the open set $\{x\mid\ell_x(1)>0\}$, the map $x\mapsto(x,\lambda_x)$ is continuous into $E_{|X|}$, so Lemma~\ref{lem:factor-continuity} gives continuity of $N_p$. At constant paths, use $N_p(x)(\I)=x(\I)$ and Lemma~\ref{lem:constant-path-continuity}. Proposition~\ref{lem:normalization-identities} gives directedness and endpoints; applying $k_\Delta$ gives the last assertion.
\end{proof}

The following interpolation is used by Morse in \cite[proof of Theorem~2, equation~\textup{(13)}]{MorseSpecialParametrization-full}. We use it both to deform paths and to lift homotopies.

\begin{lemma}[Interpolation]
\label{lem:progress-interpolation}
Let $\mathcal L\subseteq\Cont_\co(\I,[0,\infty[)$ consist of the nondecreasing functions $h$ with $h(0)=0$. For $h\in\mathcal L$ and $s\in\I$ such that $A_h(s)=s+(1-s)h(1)>0$, put
\[
\rho_{h,s}(t)=\frac{s t+(1-s)h(t)}{A_h(s)}.
\]
The map $(h,s)\mapsto\rho_{h,s}$ is continuous into $\Rep$ on its domain of definition. If $s>0$, then $\rho_{h,s}$ is an increasing homeomorphism. If $h(1)>0$, then $\rho_{h,0}=h/h(1)$; for every $h$, $\rho_{h,1}=\id_\I$. Finally, $h(t)=t h(1)$ implies $\rho_{h,s}=\id_\I$ whenever $A_h(s)>0$.
\end{lemma}
\begin{proof}
For $a<b$,
\[
\rho_{h,s}(0)=0,\qquad \rho_{h,s}(1)=1,\qquad
\rho_{h,s}(b)-\rho_{h,s}(a)\geq\frac{s(b-a)}{A_h(s)}\geq0.
\]
The last inequality is strict when $s>0$. Continuity in $(h,s)$ follows from the exponential law on $A_h(s)>0$; the remaining identities are substitutions in the defining formula.
\end{proof}

\begin{theorem}[Normalization]
\label{thm:timed-normalization}
For every timed space $p:X\to C$ and every pair $u,v\in|X|$, the subspace
\[
\Nat_p(X)(u,v)=N_p\bigl(\dP(X)(u,v)\bigr)
\]
with its relative topology is $\Delta$-generated and is a strong deformation retract of $\dP(X)(u,v)$. The deformation preserves every trace class.
\end{theorem}
\begin{proof}
Define
\begin{equation}
\label{eq:timed-homotopy}
H_p(x,s)(t)=N_p(x)\bigl((1-s)\lambda_x(t)+st\bigr).
\end{equation}
Lemma~\ref{lem:progress-interpolation} gives $\sigma_s=(1-s)\lambda_x+s\id_\I\in\Rep$, hence directedness and preservation of endpoints, images and traces. Continuity holds on nonconstant paths by composition, and at constant paths by Lemma~\ref{lem:constant-path-continuity}. Lemma~\ref{lem:interval-product} then gives continuity after $\Delta$-kelleyfication.

Proposition~\ref{lem:normalization-identities} gives
\[
H_p(x,0)=N_p(x)\lambda_x=x,\qquad
H_p(x,1)=N_p(x),\qquad
H_p(N_p(x),s)=N_p(x).
\]
The image of the continuous idempotent $N_p$ is a retract, hence a quotient of the $\Delta$-generated space $\dP(X)(u,v)$.
\end{proof}

Figure~\ref{fig:normalization-progress} illustrates the interpolation in \eqref{eq:timed-homotopy}. The same homotopy will give regular execution paths in the globular clock.

\begin{figure}[htbp]
\centering
\begin{tikzpicture}[>=stealth,font=\small]
 \begin{scope}[x=6cm,y=3.8cm]
  \fill[gray!10] (.35,0) rectangle (.65,1);
  \draw[->] (0,0)--(1.08,0) node[right] {$t$};
  \draw[->] (0,0)--(0,1.12) node[above] {progress};
  \draw[gray,dashed] (0,0)--(1,1);
  \draw[very thick] (0,0)--(.35,.5)--(.65,.5)--(1,1);
  \draw[thick,dark-red] (0,0)--(.35,.425)--(.65,.575)--(1,1);
  \draw[gray,dotted] (.35,0)--(.35,.5);
  \draw[gray,dotted] (.65,0)--(.65,.5);
  \draw (.35,.015)--(.35,-.015) node[below] {$a$};
  \draw (.65,.015)--(.65,-.015) node[below] {$b$};
  \node[below left] at (0,0) {$0$};
  \node[below] at (1,0) {$1$};
  \node[left] at (0,1) {$1$};
  \node[font=\scriptsize] at (.5,.14) {stop interval};
 \end{scope}
 \draw[very thick] (7.2,3.3)--(8,3.3);
 \node[anchor=west] at (8.15,3.3) {$s=0:\ \lambda_x$};
 \draw[thick,dark-red] (7.2,2.45)--(8,2.45);
 \node[anchor=west] at (8.15,2.45) {$s=\tfrac12:\ \tfrac12\lambda_x+\tfrac12\id_\I$};
 \draw[gray,dashed] (7.2,1.6)--(8,1.6);
 \node[anchor=west] at (8.15,1.6) {$s=1:\ \id_\I$};
 \node[anchor=west,align=left] at (7.2,.5)
 {$H_p(x,s)=N_p(x)\circ\sigma_s$\\[3pt]
  $\sigma_s=(1-s)\lambda_x+s\id_\I$};
\end{tikzpicture}
\caption{The graph of $\lambda_x$ for a nonconstant path with a stop interval $[a,b]$. For every $s>0$, the map $\sigma_s$ is strictly increasing, so $H_p(x,s)$ is a regular path with the same trace. At $s=1$ it is the normalized path $N_p(x)$.}
\label{fig:normalization-progress}
\end{figure}

A map $q:P\to T$ in $\Top$ is a \emph{Hurewicz fibration} if every homotopy lifting problem admits a lift \cite{Barthel-Riel-full}:
\[
\begin{tikzcd}[column sep=large,row sep=large]
Z\arrow[r,"a"]\arrow[d,"i_0"']&P\arrow[d,"q"]\\
Z\times\I\arrow[r,"B"']\arrow[ur,dashed,"\widetilde B"]&T.
\end{tikzcd}
\qquad Z\in\Top,\qquad i_0(z)=(z,0).
\]
A \emph{trivial Hurewicz fibration} is a Hurewicz fibration that is a homotopy equivalence. It is a trivial q-fibration, since Hurewicz fibrations are Serre fibrations and homotopy equivalences are weak homotopy equivalences.

For a map $q:P\to T$, consider the mapping path space
\[
\Gamma_q=P\times_T T^\I
=\{(x,\gamma)\mid\gamma(0)=q(x)\},
\]
with mapping spaces and pullbacks in $\Top$:
\[
\begin{tikzcd}[column sep=large,row sep=large]
\Gamma_q\arrow[r,"\pi_2"]\arrow[d,"\pi_1"']\arrow[dr,phantom,"\lrcorner",very near start]
&T^\I\arrow[d,"\operatorname{ev}_0"]\\
P\arrow[r,"q"']&T.
\end{tikzcd}
\]
A \emph{lifting function} for $q$ is a continuous section of $P^\I\to\Gamma_q$, $\beta\mapsto(\beta(0),q\beta)$. By the exponential law, it is equivalent to a lift in the square
\[
\begin{tikzcd}[column sep=large,row sep=large]
\Gamma_q\arrow[r,"\pi_1"]\arrow[d,"i_0"']&P\arrow[d,"q"]\\
\Gamma_q\times\I\arrow[r,"\operatorname{ev}"']\arrow[ur,dashed,"L"]&T,
\end{tikzcd}
\]
where $i_0(x,\gamma)=((x,\gamma),0)$ and $\operatorname{ev}((x,\gamma),s)=\gamma(s)$. Hurewicz's lifting-function criterion states that $q$ is a Hurewicz fibration if and only if this square admits a lift \cite[Sections~1--2]{HurewiczFiberSpace}. This is the universal lifting problem of \cite[Remark~5.11]{Barthel-Riel-full}, where $\Gamma_q$ is denoted by $Nq$.

Following \cite[Section~1.5]{DoldPartitions}, a map $q:P\to T$ is \emph{shrinkable} if it admits a section $j:T\to P$ and a homotopy $H:P\times\I\to P$ over $T$ from $\id_P$ to $jq$, that is,
\[
H(x,0)=x,\qquad H(x,1)=jq(x),\qquad qH(x,s)=q(x).
\]
Thus $j$ is a homotopy inverse of $q$. The following sufficient condition constructs both a lifting function and such a fiberwise contraction.

\begin{proposition}[A sufficient condition for a trivial Hurewicz fibration]
\label{prop:abstract-hurewicz}
Let $q:P\to T$ and $j:T\to P$ be continuous maps of $\Delta$-generated spaces such that $qj=\id_T$, and put $N=jq$. Give
\[
 D_q=k_\Delta\left(\bigl(P\times T\times]0,1]\bigr)
 \cup\{(x,q(x),0)\mid x\in P\}\right)
\]
the $\Delta$-kelleyfication of the relative topology in the ordinary product $P\times T\times\I$. Suppose there is a continuous map $K:D_q\to P$ such that
\begin{align}
 qK(x,\tau,s)&=\tau,\label{eq:abstract-lift-fibre}\\
 K(x,q(x),0)&=x,\label{eq:abstract-lift-initial}\\
 K(x,\tau,1)&=j(\tau).\label{eq:abstract-lift-final}
\end{align}
Then $q$ is a trivial Hurewicz fibration. The formula $H'(x,s)=K(x,q(x),s)$ gives a homotopy from $\id_P$ to $N$ satisfying $qH'(x,s)=q(x)$. If moreover $K(j(\tau),\tau,s)=j(\tau)$ for every $\tau\in T$ and $s\in\I$, then $H'$ is a strong deformation retraction onto $j(T)$.
\end{proposition}
\begin{proof}
The continuous evaluation map
\[
\Gamma_q\times\I\longrightarrow D_q,
\qquad ((x,\gamma),s)\longmapsto(x,\gamma(s),s)
\]
and $K$ give a lift in the universal square:
\[
\begin{tikzcd}[column sep=large,row sep=large]
\Gamma_q\arrow[r,"\pi_1"]\arrow[d,"i_0"']&P\arrow[d,"q"]\\
\Gamma_q\times\I\arrow[r,"\operatorname{ev}"']\arrow[ur,dashed,"L"]&T.
\end{tikzcd}
\qquad
\begin{aligned}
L((x,\gamma),s)&=K(x,\gamma(s),s),\\
qL((x,\gamma),s)&=\gamma(s),\\
L((x,\gamma),0)&=x.
\end{aligned}
\]
Continuity into $D_q$ uses Lemma~\ref{lem:interval-product} and the defining property of $k_\Delta$. For an arbitrary homotopy lifting problem, the lift is
\begin{equation}
\label{eq:abstract-homotopy-lift}
\begin{tikzcd}[column sep=large,row sep=large]
Z\arrow[r,"a"]\arrow[d,"i_0"']&P\arrow[d,"q"]\\
Z\times\I\arrow[r,"B"']\arrow[ur,dashed,"\widetilde B"]&T,
\end{tikzcd}
\qquad
\widetilde B(z,s)=K\bigl(a(z),B(z,s),s\bigr).
\end{equation}
For constant base paths,
\[
H'(x,s)=K(x,q(x),s),\qquad
H'(x,0)=x,\qquad H'(x,1)=jq(x),\qquad qH'(x,s)=q(x).
\]
Thus $q$ is shrinkable and a homotopy equivalence. The additional identity makes $H'$ fix $j(T)$.
\end{proof}

\begin{theorem}[Paths and traces of timed spaces]
\label{thm:timed-traces}
For every timed space $p:X\to C$ and every $u,v\in|X|$, restriction of the quotient induces a homeomorphism
\[
\Nat_p(X)(u,v)\xrightarrow{\ \cong\ }\dT(X)(u,v).
\]
The quotient $q_{u,v}$ is a trivial Hurewicz fibration and, in particular, a trivial q-fibration. If $|X|$ is $\Delta$-Hausdorff, its trace spaces are $\Delta$-Hausdorff as well.
\end{theorem}
\begin{proof}
Write $P=\dP(X)(u,v)$, $T=\dT(X)(u,v)$, and $q=q_{u,v}$. Continuity of $N_p$, its constancy on trace classes, and the quotient property give
\[
\begin{tikzcd}[column sep=large]
P\arrow[r,"q"]\arrow[dr,"N_p"']&T\arrow[d,"j"]\\
&P.
\end{tikzcd}
\qquad jq=N_p,\qquad qj=\id_T.
\]
Thus $j$ and $q|_{\Nat_p(X)(u,v)}$ are inverse homeomorphisms.

For $(x,\tau,s)\in D_q$, put $D(x,s)=s+(1-s)\ell_x(1)$. Where $D(x,s)>0$, Lemma~\ref{lem:progress-interpolation} gives
\begin{equation}
\label{eq:homotopy-lift-parameter}
\eta_{x,s}(t)=\rho_{\ell_x,s}(t)=\frac{s t+(1-s)\ell_x(t)}{s+(1-s)\ell_x(1)}.
\end{equation}
Define
\begin{equation}
\label{eq:homotopy-lift}
K(x,\tau,s)=
\begin{cases}
j(\tau)\circ\eta_{x,s},&D(x,s)>0,\\
x,&D(x,s)=0.
\end{cases}
\end{equation}
The denominator vanishes precisely at $s=0$ with $x$ constant, where $\tau=q(x)$. Proposition~\ref{lem:normalization-identities} and Lemma~\ref{lem:progress-interpolation} give
\[
\begin{aligned}
qK(x,\tau,s)&=\tau,&K(x,q(x),0)&=N_p(x)\lambda_x=x,\\
K(x,\tau,1)&=j(\tau),&K(j(\tau),\tau,s)&=j(\tau).
\end{aligned}
\]
Continuity where $D>0$ follows by composition. At $D=0$, use Lemma~\ref{lem:constant-path-continuity} with $G(x,\tau,s)=j(\tau)$ and
\begin{equation}
\label{eq:homotopy-lift-images}
K(x,\tau,s)(\I)=j(\tau)(\I).
\end{equation}
Applying $k_\Delta$ gives $K:D_q\to P$ continuous, and Proposition~\ref{prop:abstract-hurewicz} applies.

If $|X|$ is $\Delta$-Hausdorff, apply \cite[Propositions~B.5--B.6]{leftproperflow-full} to the continuous injection
\[
T\xrightarrow{j}P\longrightarrow\Cont_\Delta(\I,|X|)
\]
to conclude that $T$ is $\Delta$-Hausdorff.
\end{proof}

For paths in $\mathbb R^n$ with $n\ge2$, Hoehn, Oversteegen and Tymchatyn construct a \emph{standard parameterization} using their length function \cite[Section~4, Definitions~1--2 and Theorem~4.1]{HoehnOversteegenTymchatyn}. Their Fr\'echet metric induces the quotient topology, and standard parameterization identifies this quotient with a closed subspace of $C(\I,\mathbb R^n)$. The present normalization applies to path families in timed spaces, and the map $K$ also gives the homotopy lifting property.

\begin{corollary}
\label{cor:submetrizable}
If a saturated directed space has submetrizable underlying space, every quotient map from directed paths with fixed endpoints to traces is a trivial Hurewicz fibration.
\end{corollary}
\begin{proof}
The identity map makes this directed space a timed space over itself.
\end{proof}

\begin{remark}
Normalization depends on the chosen metric injection and the triple $\mathfrak a$. The conclusions about local presentability, the homeomorphism type of traces, and the trivial Hurewicz fibration are independent of these choices.
\end{remark}

\begin{corollary}[Invariant path families]
\label{cor:invariant-family}
Let $\mathcal A\subseteq\Cont_\co(\I,Z)$ be a family of paths with fixed endpoints. Suppose either that $Z$ is submetrizable, or that $Z$ is the underlying space of a timed space and the paths in $\mathcal A$ are directed. Construct $\lambda_x$, $N(x)$ and $H(x,s)$ using, respectively, a continuous injective map into a metric space or the progress function of the timed space. Suppose $N$ and $H$ take their values in $\mathcal A$.

If $\mathcal A$ is stable under $\Rep$, its quotient map by those reparametrizations, with the quotient topology of $k_\Delta\mathcal A$, is a trivial Hurewicz fibration. It admits a continuous section and a strong deformation retraction onto the section's image, preserving quotient classes. The same conclusions hold for the group $\operatorname{Homeo}_+(\I)$ of increasing homeomorphisms when $\mathcal A$ consists of regular paths and is stable under this group.
\end{corollary}
\begin{proof}
Theorem~\ref{thm:multiscale-aggregation} or \eqref{eq:timed-equivariance}--\eqref{eq:timed-stops} gives a continuous function $\ell$ on $\mathcal A\times\I$ with the stated equivariance whenever the reparametrized path belongs to $\mathcal A$. The assumed invariance replaces saturation in the proofs of Proposition~\ref{prop:normalization-continuous} and Theorem~\ref{thm:timed-normalization}. With $P=k_\Delta\mathcal A$ and $T$ its quotient, the same factorization gives $jq=N$, $qj=\id_T$ and the strong deformation retraction $H$ preserving quotient classes. Stability under $\Rep$ keeps \eqref{eq:homotopy-lift} in $\mathcal A$, so the proof of Theorem~\ref{thm:timed-traces} applies.

In the group case, regularity gives
\[
x\text{ nonconstant}\quad\Longrightarrow\quad
\lambda_x,\ (1-s)\lambda_x+s\id_\I,\ \eta_{x,s}\in\operatorname{Homeo}_+(\I)
\]
whenever $\eta_{x,s}$ is defined. For constant paths, $\lambda_x=\id_\I$ and $\eta_{x,s}=\id_\I$ for $s>0$. Thus all constructions stay in $\mathcal A$ and preserve its orbits.
\end{proof}

Table~\ref{tab:roles-of-axioms} summarizes how the clock axioms enter the categorical and homotopical arguments.

\begin{table}[htbp]
\centering
\small
\renewcommand{\arraystretch}{1.3}
\begin{tabular}{@{}p{.32\linewidth}@{\hspace{1.4em}}p{.62\linewidth}@{}}
\hline
\textbf{Hypothesis or construction} & \textbf{Role in the argument}\\
\hline
Submetrizability of the clock $C$ & A continuous injective map into a metric space gives the progress function of Theorem~\ref{thm:multiscale-aggregation}.\\
Regularity of $p:X\to C$ & Progress is constant on an interval if and only if the original directed path is constant there.\\
Saturation of $X$ & The factor $N_p(x)$ in $x=N_p(x)\lambda_x$ is again directed.\\
Continuous monotone factorization & Normalization is continuous in the original topology of $X$, even when the auxiliary metric topology is coarser.\\
\raggedright Proposition~\ref{prop:abstract-hurewicz} and formula~\eqref{eq:homotopy-lift-parameter}\par & Every homotopy of traces lifts with prescribed initial paths; preservation of images proves continuity at constant paths.\\
\raggedright Small orthogonality class\par & The category $\Timed(C)$ is locally presentable; this step does not require submetrizability.\\
\hline
\end{tabular}
\medskip \caption{The distinct roles of the clock axioms and the two main arguments. See Theorem~\ref{thm:timed-locally-presentable}, Lemma~\ref{lem:factor-continuity}, and Theorem~\ref{thm:timed-traces}.}
\label{tab:roles-of-axioms}
\end{table}

\section{The cubical clock}
\label{sec:cubical}

The directed circle
\[
\dS=\mathbb R/\mathbb Z
\]
has as directed paths those admitting a continuous nondecreasing lift to $\mathbb R$. We call it the \emph{cubical clock}. A timed space over $\dS$ is a regular clock map in the sense of \cite{clockmap}.

The following descent lemma applies both to cubical paths and to globular naturalization.

\begin{lemma}
\label{lem:descent-reparametrization}
Let $J,J'$ be compact real intervals, let $\rho:J\to J'$ be a continuous nondecreasing surjection, and let $P$ be a finite product of compact real intervals with the coordinatewise order. Suppose that $g:P\to Y$ is continuous and has totally disconnected fibers. If $\theta:J\to P$ is continuous and nondecreasing and $x:J'\to Y$ satisfies $x\rho=g\theta$, then there is a unique continuous nondecreasing map $\bar\theta:J'\to P$ with $\theta=\bar\theta\rho$. Moreover, $x=g\bar\theta$, and $\bar\theta$ is surjective if $\theta$ is surjective.
\end{lemma}
\begin{proof}
For $s\in J'$, the set $\theta(\rho^{-1}(s))$ is connected and lies in $g^{-1}(x(s))$, so it is a singleton. The quotient property of $\rho$ gives the unique continuous factor \cite[Theorem~22.2]{zbMATH01461253}:
\[
\begin{tikzcd}[column sep=large,row sep=large]
J\arrow[r,"\theta"]\arrow[d,two heads,"\rho"']&P\arrow[d,"g"]\\
J'\arrow[r,"x"']\arrow[ur,dashed,"\bar\theta"]&Y.
\end{tikzcd}
\]
For $s<t$, choose $u\in\rho^{-1}(s)$ and $v\in\rho^{-1}(t)$. Then
\[
 u<v\quad\Longrightarrow\quad
 \bar\theta(s)=\theta(u)\leq\theta(v)=\bar\theta(t).
\]
Surjectivity of $\rho$ gives $x=g\bar\theta$ and $\bar\theta(J')=\theta(J)$.
\end{proof}

The circle is metrizable. For saturation, let $x\rho$ be directed, where $x:\I\to\dS$ is continuous and $\rho\in\Rep$. Its nondecreasing lift takes values in a compact interval $P\subset\mathbb R$. The restricted covering map has finite fibers, so Lemma~\ref{lem:descent-reparametrization} supplies the nondecreasing lift $\bar\theta$:
\[
\begin{tikzcd}[column sep=large,row sep=large]
 \I\arrow[r,"\theta"]\arrow[d,two heads,"\rho"']&P\arrow[d]\\
 \I\arrow[r,"x"']\arrow[ur,dashed,"\bar\theta"]&\dS.
\end{tikzcd}
\]

A \emph{precubical set} $K$ consists of sets $(K_n)_{n\geq0}$ and face maps $\partial_i^\epsilon:K_n\to K_{n-1}$ satisfying
\[
\partial_i^\epsilon\partial_j^\eta=\partial_{j-1}^\eta\partial_i^\epsilon,
\qquad 1\leq i<j\leq n,\quad \epsilon,\eta\in\{0,1\}.
\]
A morphism is a family of maps commuting with the face maps. The realization is
\[
|K|=\left(\coprod_{n\geq0}K_n\times\I^n\right)/{\sim},
\qquad (\partial_i^\epsilon c,t)\sim(c,\delta_i^\epsilon t),
\]
where each $K_n$ is discrete and $\delta_i^\epsilon$ inserts $\epsilon$ in the $i$th coordinate. This is a CW-complex with one open cell for each cube, hence belongs to either choice of $\Top$. Write $\chi_c(t)=[c,t]$ for the characteristic map of $c\in K_n$. The standard directed paths are finite concatenations of paths $\chi_c\gamma$ with $\gamma$ nondecreasing in each coordinate; see \cite[Definitions~2.1--2.2]{MR2521708}.

\begin{lemma}
\label{lem:cubical-saturated}
The standard directed realization of every precubical set is saturated.
\end{lemma}
\begin{proof}
The fibers of $\chi_c$ are finite: each relative open face maps injectively onto its open cell, and a cube has finitely many faces.

Let $x:\I\to|K|$ be continuous and suppose $x\rho$ is directed for $\rho\in\Rep$. Choose a subdivision $0=a_0<\cdots<a_r=1$ and maps $\gamma_i:[a_{i-1},a_i]\to\I^{n_i}$ which are nondecreasing in each coordinate and satisfy
\[
(x\rho)|_{[a_{i-1},a_i]}=\chi_{c_i}\gamma_i.
\]
With $b_i=\rho(a_i)$, Lemma~\ref{lem:descent-reparametrization} gives continuous coordinatewise nondecreasing maps $\delta_i$ such that
\[
 \gamma_i=\delta_i\circ\rho|_{[a_{i-1},a_i]},
 \qquad
 x|_{[b_{i-1},b_i]}=\chi_{c_i}\delta_i.
\]
Omit the degenerate intervals $[b_{i-1},b_i]$ to obtain a directed presentation of $x$.
\end{proof}

The following formula is the map $s$ introduced in \cite[Section~2.2.1, p.~1721, after Remark~2.5]{MR2521708}. We check its regularity in the present terminology.

\begin{proposition}[Raussen's formula]
\label{prop:cubical-timing}
For every precubical set $K$, the formula
\[
p_K\chi_c(t_1,\ldots,t_n)=t_1+\cdots+t_n\pmod{\mathbb Z}
\]
defines a regular directed map $p_K:|K|\to\dS$. Thus $(|K|,p_K)$ is a timed space over the cubical clock. This construction is natural in $K$.
\end{proposition}
\begin{proof}
On a face with inserted coordinate $\epsilon\in\{0,1\}$, the coordinate sum is
\[
\sum_{j=1}^{n}(\delta_i^\epsilon t)_j
=\sum_{j=1}^{i-1}t_j+\epsilon+\sum_{j=i}^{n-1}t_j
=\epsilon+\sum_{j=1}^{n-1}t_j
\equiv\sum_{j=1}^{n-1}t_j\pmod{\mathbb Z}.
\]
Thus the formulas descend to a continuous map. The coordinate sum is a nondecreasing lift on every cubical path, so $p_K$ is directed.

For regularity, Lemma~\ref{lem:regular-generators} reduces the argument to a cubical path $\chi_c\gamma$. If $p_K\chi_c\gamma$ is constant on $[a,b]$, continuity makes $\sum_j\gamma_j$ constant there. Hence, for $a\leq s<t\leq b$,
\[
 0=\sum_{j=1}^n\bigl(\gamma_j(t)-\gamma_j(s)\bigr),
 \quad \gamma_j(t)-\gamma_j(s)\geq0
 \quad\Longrightarrow\quad \gamma(t)=\gamma(s).
\]
Lemma~\ref{lem:cubical-saturated} gives saturation. For a morphism $g:K\to K'$,
\[
p_{K'}|g|\chi_c(t)=p_{K'}\chi_{g(c)}(t)
=\left[\sum_{j=1}^nt_j\right]=p_K\chi_c(t),
\]
so the construction is natural.
\end{proof}

\begin{figure}[htbp]
\centering
\begin{tikzpicture}[scale=2.1,>=stealth]
\draw[->] (-0.08,0)--(1.18,0) node[right] {$t_1$};
\draw[->] (0,-0.08)--(0,1.18) node[above] {$t_2$};
\draw (1,0)--(1,1)--(0,1);
\draw[dashed,gray] (0,0.4)--(0.4,0);
\draw[dashed,gray] (0,0.8)--(0.8,0);
\draw[dashed,gray] (0.2,1)--(1,0.2);
\draw[dashed,gray] (0.6,1)--(1,0.6);
\draw[very thick,->] (0,0)--(0.3,0.2)--(0.65,0.25)--(0.75,0.7)--(1,1);
\node[below left] at (0,0) {$0$};
\node[above right] at (1,1) {$2$};
\node at (0.5,-0.32) {$t_1+t_2$};
\draw[->,thick] (1.4,0.52)--(2.02,0.52) node[midway,above] {$\bmod\,\mathbb Z$};
\draw (2.57,0.52) circle (0.43);
\draw[thick,->] (2.57,0.95) arc[start angle=90,end angle=230,radius=0.43];
\fill (3,0.52) circle (0.025);
\node[right] at (3,0.52) {$[0]$};
\node at (2.57,-0.14) {$\dS$};
\end{tikzpicture}
\caption{The sum of cubical coordinates measures progress on each cube. Reduction modulo integers makes the functions agree on all face identifications.}
\label{fig:cubical-clock}
\end{figure}

\begin{corollary}
For every precubical set $K$ and every $u,v\in|K|$, the quotient $q_{u,v}$ is a trivial Hurewicz fibration. The category $\Timed(\dS)$ is locally presentable and contains all these realizations with their canonical maps to $\dS$.
\end{corollary}
\begin{proof}
Apply Proposition~\ref{prop:cubical-timing} and Theorems~\ref{thm:timed-locally-presentable} and~\ref{thm:timed-traces}.
\end{proof}

\Needspace{6\baselineskip}
\begin{remark}
The realization of every precubical set is a CW-complex, so Proposition~\ref{prop:cellular-submetrizable} shows that its underlying space is submetrizable. By Lemma~\ref{lem:cubical-saturated}, its standard directed structure is saturated. Thus this realization is itself a saturated clock, and its identity map defines a timed space over itself.
\end{remark}

Raussen's naturalization identifies traces with naturally parametrized paths and uses the linear interpolation in \eqref{eq:timed-homotopy}; see \cite[Propositions~2.15--2.16]{MR2521708}. The extension suggested in \cite[Remark~2.17]{MR2521708} uses a continuous additive length function; our progress function need not be additive. The homotopy equivalence for regular maps to the circle is proved in \cite[Theorem~6.7]{clockmap}. The corollary establishes the homotopy lifting property for the same quotient.

\section{Globular clocks and execution paths}
\label{sec:globular}

\subsection{Cellular realizations are saturated clocks}

A \emph{cellular space} is obtained from a discrete space by transfinite attachments of closed disks of finite dimension along their boundaries, with the final topology at limit stages; see \cite[Section~4, Definition~9 and Proposition~16]{Moore3}.

\begin{proposition}
\label{prop:cellular-submetrizable}
Every cellular space admits a continuous injective map into a Hilbert space. Consequently every retract of a cellular space is submetrizable.
\end{proposition}

\begin{proof}
Let $\mathcal H$ be the Hilbert direct sum of one copy of $\mathbb R$ for every initial point and one copy of $\mathbb R^{m+1}$ for every attached $m$-disk. Map the initial points to their unit basis vectors and extend by transfinite induction, using only the summands assigned to the current or preceding stages.

For an attachment $a:\mathsf S^{m-1}\to Z$, $m\geq1$, and a previously constructed injection $F:Z\to\mathcal H$, put $b=Fa$. In the preceding summands and the new summand, respectively, define
\[
 B(z)=
 \begin{cases}
 \|z\|\,b(z/\|z\|),&z\ne0,\\
 0,&z=0.
 \end{cases}
 \qquad
 v(z)=(1-\|z\|)(1,z)\in\mathbb R^{m+1}.
\]
Compactness of $\mathsf S^{m-1}$ gives $C=\sup_u\|b(u)\|<\infty$ and
\[
 \|B(z)-B(0)\|\leq C\|z\|\longrightarrow0
 \quad(z\to0),
 \qquad (B+v)|_{\mathsf S^{m-1}}=b.
\]
The extension is injective because its new component vanishes on $Z$ and determines every interior point:
\[
 \|z\|<1\quad\Longrightarrow\quad
 v(z)=(r,u),\quad r=1-\|z\|>0,\quad z=u/r.
\]
\[
\begin{tikzcd}[column sep=large,row sep=large]
 \mathsf S^{m-1}\arrow[r,"a"]\arrow[d,hook]&Z\arrow[d,hook]\arrow[ddr,bend left=20,"F"]&\\
 \mathsf D^m\arrow[r]\arrow[drr,bend right=15,"B+v"']&Z\sqcup_a\mathsf D^m\arrow[ul,phantom,"\ulcorner",very near start]\arrow[dr,dashed]&\\
 &&\mathcal H
\end{tikzcd}
\]

For a $0$-disk, use a new unit coordinate. At a limit stage the union is continuous by the final topology and remains injective. For a retract $A\xrightarrow{i}Z\xrightarrow{r}A$, $ri=\id_A$, the composite $Fi$ is a continuous injection.
\end{proof}

\begin{remark}
For $m\geq1$, the absolute retract property of Hilbert spaces for metrizable spaces \cite[Theorem~4.1]{Dugundji} supplies an extension $B$ of $Fa$ with values in the preceding summands. Let $q_m:\mathsf D^m\to\mathsf D^m/\mathsf S^{m-1}$ be the quotient map. Since $\mathsf D^m/\mathsf S^{m-1}\cong\mathsf S^m$, choose an embedding $j_m:\mathsf D^m/\mathsf S^{m-1}\hookrightarrow\mathbb R^{m+1}$ sending the collapsed boundary to $0$. Then
\[
v=j_mq_m,\qquad v^{-1}(0)=\mathsf S^{m-1},\qquad v|_{\operatorname{int}\mathsf D^m}\text{ is injective}.
\]
The map $B+v$ induces a continuous injection on the pushout, and the remaining induction applies.
\end{remark}

\begin{remark}
The map need not be an embedding: submetrizability requires a coarser metrizable topology, rather than metrizability of the given topology. For CW-complexes, the result follows from stratifiability \cite[Section~1.3, Exercise~2, p.~33]{MR1074175} and the submetrizability of stratifiable spaces \cite[Introduction, (A) and (C), and Lemma~8.2]{BorgesStratifiable}. Radial extension and coordinates that vanish on cell boundaries are used in \cite[proof of Theorem~1.5.15, pp.~46--48]{MR1074175} for locally finite countable CW-complexes of finite dimension. The construction above allows arbitrary transfinite attachments, without a restriction on the order of the dimensions.
\end{remark}

Let $\mathcal M$ be the reparametrization category whose objects are the positive real numbers and whose maps $\mathcal M(\ell,\ell')$ are the continuous nondecreasing surjections $[0,\ell]\to[0,\ell']$. Let $\mathcal G\subset\mathcal M$ have the same objects and only the increasing homeomorphisms as maps. Composition is composition of continuous maps, and each space of maps has the $\Delta$-kelleyfication of its relative compact-open topology. In particular, $\mathcal M(1,1)=\Rep$ and $\mathcal G(1,1)=\operatorname{Homeo}_+(\I)$.

For $\mathcal P\in\{\mathcal G,\mathcal M\}$, a \emph{$\mathcal P$-multipointed $d$-space} $A$ is a triple $(|A|,A^0,\mathbb P^{\mathrm{top}}A)$ consisting of an underlying space $|A|\in\Top$, a set of states $A^0\subset|A|$, and a set $\mathbb P^{\mathrm{top}}A$ of continuous maps $\I\to|A|$, called \emph{execution paths}, satisfying the following axioms:
\begin{enumerate}
\item Every execution path $x$ satisfies $x(0),x(1)\in A^0$.
\item If $x$ is an execution path and $\phi\in\mathcal P(1,1)$, then $x\phi$ is an execution path.
\item If $x,y$ are execution paths with $x(1)=y(0)$, then their normalized composition $x*_Ny$ is an execution path.
\end{enumerate}
Constant paths are not required to be execution paths. A morphism $f:A\to B$ is a continuous map $|f|:|A|\to|B|$ such that $|f|(A^0)\subset B^0$ and $|f|\circ x\in\mathbb P^{\mathrm{top}}B$ for every $x\in\mathbb P^{\mathrm{top}}A$. The resulting category is denoted by $\mathcal P\mathsf{dTop}$; see \cite[Section~3]{Moore3}. We use the q-model structures of that paper. In particular, a q-cofibrant object is a retract of a cellular object.

For $\alpha,\beta\in A^0$, the space $\mathbb P^{\mathrm{top}}_{\alpha,\beta}A$ consists of the execution paths from $\alpha$ to $\beta$, with the $\Delta$-kelleyfication of the relative compact-open topology.

For a space $Z$, the topological globe $\Glob_{\mathcal P}(Z)$ has underlying space
\[
\bigl(\{0,1\}\sqcup(Z\times\I)\bigr)/\bigl((z,0)\sim0,\ (z,1)\sim1\bigr),
\]
states $0,1$, and execution paths $t\mapsto[z,\phi(t)]$ for $z\in Z$ and $\phi\in\mathcal P(1,1)$. For fixed $z$, the path $t\mapsto[z,t]$ is called a \emph{meridian} of the globe.

The generating q-cofibrations also include $\varnothing\to\{0\}$ and $\{0,1\}\to\{0\}$ between discrete state spaces. The following observation allows us to place these operations before the globular attachments; see \cite[Section~4]{GlobularNaturalSystem-full}.

\begin{lemma}
\label{lem:globular-presentation}
For $\mathcal P\in\{\mathcal G,\mathcal M\}$, every cellular $\mathcal P$-multipointed $d$-space admits a presentation starting from its final discrete set of states and using only globular attachments. Every q-cofibrant object is a retract of an object with such a presentation.
\end{lemma}
\begin{proof}
Let $S$ be the final set of states of a cellular decomposition and put
\[
\widehat A_\nu=A_\nu\sqcup_{A_\nu^0}S,
\]
\[
\begin{tikzcd}[column sep=large,row sep=large]
 A_\nu^0\arrow[r]\arrow[d]&A_\nu\arrow[d]\\
 S\arrow[r]&\widehat A_\nu\arrow[ul,phantom,"\ulcorner",very near start],
\end{tikzcd}
\]
where $A_\nu^0$ and $S$ are discrete multipointed $d$-spaces. Commutation of pushouts turns the state maps into identities and leaves globular attachments. If the original colimit is $A$, then
\[
 \mathop{\mathrm{colim}}_\nu\widehat A_\nu
 =A\sqcup_{A^0}S=A,
 \qquad A^0=S.
\]
The final assertion follows by taking retracts.
\end{proof}

The functor
\[
 \Sp:\mathcal M\mathsf{dTop}\longrightarrow\dTop
\]
is left adjoint to the full and faithful functor $\Om$ of \cite[Proposition~3.7 and Theorems~3.8--3.9]{GlobularNaturalSystem-full}. It preserves the underlying space and generates a directed structure from the execution paths by adding constant paths and closing under continuous nondecreasing changes of parameter and concatenation. A cellular multipointed $d$-space is constructed by globular attachments
\[
\begin{tikzcd}
 \Glob_{\mathcal M}(\mathsf{S}^{n-1})
 \arrow[r]\arrow[d,hook]
 & A_\nu\arrow[d]\\
 \Glob_{\mathcal M}(\mathsf{D}^n)
 \arrow[r]&A_{\nu+1}\arrow[ul,phantom,"\ulcorner",very near start].
\end{tikzcd}
\]
On underlying spaces, these are disk attachments of dimension $n+1$. These pushouts and colimits, calculated in the category of topological spaces, are $\Delta$-generated. Proposition~\ref{prop:cellular-submetrizable} gives a continuous injective map from each of these spaces to a Hilbert space. They are therefore Hausdorff and belong to either choice of $\Top$. Thus these colimits are also colimits in $\Top$, and Proposition~\ref{prop:cellular-submetrizable} applies to $|A|$.

\begin{proposition}
\label{lem:change-reparametrizations}
For a $\mathcal G$-multipointed $d$-space $A$, let $A^{\mathcal M}$ have the same underlying space and states and execution paths
\[
 \{x\phi\mid x\in\mathbb P^{\mathrm{top}}A,\ \phi\in\Rep\}.
\]
This construction is left adjoint to the forgetful functor from $\mathcal M$-multipointed $d$-spaces to $\mathcal G$-multipointed $d$-spaces. It preserves cellular and q-cofibrant objects.
\end{proposition}
\begin{proof}
Closure under $\Rep$ follows from composition in $\Rep$. For normalized composition,
\[
 (x\phi)*_N(y\psi)=(x*_Ny)\theta,
 \qquad
 \theta(t)=
 \begin{cases}
 \frac12\phi(2t),&0\leq t\leq\frac12,\\
 \frac12+\frac12\psi(2t-1),&\frac12\leq t\leq1.
 \end{cases}
\]
Here $\theta\in\Rep$. A morphism from $A$ to an $\mathcal M$-multipointed $d$-space preserves this closure, which gives the adjunction. Its left adjoint fixes discrete state spaces, sends $\Glob_{\mathcal G}(Z)$ to $\Glob_{\mathcal M}(Z)$, and preserves colimits and retracts. Hence it preserves cellular and q-cofibrant objects. This is the functor $\mathrm F^{\mathcal M}_{\mathcal G}$ of \cite[Notation~21]{Moore3}.
\end{proof}

For a $\mathcal G$-multipointed $d$-space $A$, we use the notation $\Sp(A)=\Sp(A^{\mathcal M})$. This is exactly the directed structure generated by the original execution paths of $A$, since nondecreasing reparametrizations are already allowed in a directed space.

\begin{proposition}
\label{lem:globular-saturated}
If $A$ is a q-cofibrant $\mathcal M$-multipointed $d$-space, then $\Sp(A)$ is saturated. The execution paths of $A$ are also saturated: if $x\rho$ is an execution path of $A$ for a continuous $x$ and $\rho\in\Rep$, then $x$ is an execution path of $A$.
\end{proposition}

\begin{proof}
For execution paths, the assertion is \cite[Proposition~31]{Moore3}. Suppose first that $A$ is cellular, $x:\I\to|A|$ is continuous, and $x\rho$ is a nonconstant directed path for $\rho\in\Rep$. Globular naturalization \cite[Theorems~4.7 and~4.9]{GlobularNaturalSystem-full} gives
\[
 \begin{gathered}
 x\rho=\eta\theta,\qquad m\geq1,\\
 \eta:[0,m]\to|A|\text{ a nonconstant regular Moore execution path},\\
 \theta:\I\to[0,m]\text{ continuous and nondecreasing}.
 \end{gathered}
\]

Regularity makes the fibers of $\eta$ totally disconnected. Lemma~\ref{lem:descent-reparametrization} therefore gives
\[
 \begin{gathered}
 \theta=\bar\theta\rho,
 \qquad \bar\theta:\I\to[0,m]\text{ continuous and nondecreasing},\\
 x=\eta\bar\theta\in d\Sp(A).
 \end{gathered}
\]
If $x\rho$ is constant, so is $x$.

If $x\rho$ is an execution path, it is nonconstant. Choose $x\rho=\eta\theta$ from its globular naturalization \cite[Theorem~6 and Proposition~20]{Moore3}, with $\theta$ surjective and $e(u)=\eta(mu)$ an execution path. The same descent argument gives
\[
 \bar\theta(\I)=\theta(\I)=[0,m]
 \quad\Longrightarrow\quad
 \bar\theta/m\in\Rep,
 \qquad x=e\circ(\bar\theta/m)\in\mathbb P^{\mathrm{top}}A.
\]

Both assertions pass to a retract $A\xrightarrow{i}B\xrightarrow{r}A$ of a cellular object, since
\[
 ri=\id_A,\qquad (ix)\rho=i(x\rho),\qquad x=r(ix).
\]
\end{proof}

\begin{theorem}
\label{thm:cofibrant-saturated-clock}
For every q-cofibrant $\mathcal P$-multipointed $d$-space $A$, where $\mathcal P\in\{\mathcal G,\mathcal M\}$, the directed space $\Sp(A)$ is itself a saturated clock. In particular, its identity map makes it a timed space over itself.
\end{theorem}
\begin{proof}
For $\mathcal P=\mathcal M$, the underlying space is a retract of a cellular space and is therefore submetrizable by Proposition~\ref{prop:cellular-submetrizable}. Proposition~\ref{lem:globular-saturated} proves saturation. For $\mathcal P=\mathcal G$, apply this argument to $A^{\mathcal M}$, which is q-cofibrant by Proposition~\ref{lem:change-reparametrizations} and has the same underlying space as $A$. The identity map is regular in either case.
\end{proof}

\begin{lemma}
\label{lem:execution-normalization}
Let $A$ be a q-cofibrant $\mathcal P$-multipointed $d$-space, $\mathcal P\in\{\mathcal G,\mathcal M\}$, and let $p:\Sp(A)\to C$ define a timed space. For fixed metric data on $C$ and every $\alpha,\beta\in A^0$, the maps $N_p$ and $H_p$ preserve $\mathbb P^{\mathrm{top}}_{\alpha,\beta}A$ and each of its classes under $\mathcal P$-reparametrization. These execution paths are nonconstant; for $\mathcal P=\mathcal G$, they are regular.
\end{lemma}
\begin{proof}
Nonconstancy follows from the cellular carrier description \cite[Theorem~6]{Moore3}; for $\mathcal G$, regularity follows from \cite[Proposition~5.13]{Moore2-full}. Both properties pass to retracts, since their sections are injective. Put $E=\mathbb P^{\mathrm{top}}_{\alpha,\beta}A$ and $\rho_{x,s}=(1-s)\lambda_x+s\id_\I$. For $\mathcal M$, Proposition~\ref{lem:globular-saturated} gives
\[
 \begin{gathered}
 x=N_p(x)\lambda_x\in E,\quad\lambda_x\in\Rep
 \quad\Longrightarrow\quad N_p(x)\in E,\\
 \rho_{x,s}\in\Rep\quad\Longrightarrow\quad
 H_p(x,s)=N_p(x)\rho_{x,s}\in E.
 \end{gathered}
\]
For $\mathcal G$, equation~\eqref{eq:timed-stops} gives
\[
 \lambda_x,\rho_{x,s}\in\operatorname{Homeo}_+(\I),
 \qquad N_p(x)=x\lambda_x^{-1},
 \qquad H_p(x,s)=x\lambda_x^{-1}\rho_{x,s}.
\]
These factorizations also give preservation of each reparametrization class.
\end{proof}

\subsection{Regular maps to a fixed globular clock}

Fix a cellular q-cofibrant replacement $Q=\mathbf1^{\mathrm{cell}}\to\mathbf1$ of the final $\mathcal M$-multipointed $d$-space. The replacement and the metric data on $|Q|$ are fixed independently of the source of a clock map.

\begin{definition}
The \emph{globular clock} associated with this replacement is
\[
 \Cg=\Sp(Q)=\Sp(\mathbf1^{\mathrm{cell}}).
\]
\end{definition}

\begin{corollary}
\label{cor:globular-clock}
The globular clock is a saturated clock. The category $\Timed(\Cg)$ is locally presentable, and every object $p:X\to\Cg$ has a trivial Hurewicz fibration $q_{u,v}$ from each space of directed paths with fixed endpoints to its trace space.
\end{corollary}
\begin{proof}
Theorem~\ref{thm:cofibrant-saturated-clock} applies to $Q$. The conclusions follow from Theorem~\ref{thm:timed-locally-presentable} and Theorem~\ref{thm:timed-traces}.
\end{proof}

Regular maps to $\Cg$ will be constructed by extending families of execution paths over the parameter disks of globes.

\begin{proposition}
\label{lem:clock-regularization}
The object $Q$ has a unique state $*$. Let
\[
 E_Q=\mathbb P^{\mathrm{top}}_{*,*}Q.
\]
Every continuous map $\mathsf{S}^{n-1}\to E_Q$ extends to $\mathsf{D}^n\to E_Q$, for every $n\geq0$. Moreover there is a continuous map
\[
 R:E_Q\times\I\longrightarrow E_Q
\]
such that $R(\gamma,0)=\gamma$ and $R(\gamma,s)$ is a regular path for every $s>0$.
\end{proposition}
\begin{proof}
The trivial q-fibration $Q\to\mathbf1$ is bijective on states. The natural bijection for globes \cite[Proposition~13]{Moore3} identifies its lifts against the globular generating q-cofibrations \cite[Section~4]{Moore3} with the required extensions:
\[
\begin{tikzcd}[column sep=large,row sep=large]
 \Glob_{\mathcal M}(\mathsf S^{n-1})\arrow[r,"\widetilde b"]\arrow[d,hook]
   &Q\arrow[d]\\
 \Glob_{\mathcal M}(\mathsf D^n)\arrow[r]\arrow[ur,dashed,"\widetilde F"]
   &\mathbf1
\end{tikzcd}
\qquad
\begin{tikzcd}[column sep=large,row sep=large]
 \mathsf S^{n-1}\arrow[r,"b"]\arrow[d,hook]
   &E_Q\arrow[d]\\
 \mathsf D^n\arrow[r]\arrow[ur,dashed,"F"]
   &\{*\}.
\end{tikzcd}
\]
For $n=0$, this gives $E_Q\ne\varnothing$.

Apply Lemma~\ref{lem:execution-normalization} to the identity timed space on $\Cg$:
\begin{equation}
\label{eq:clock-regularization}
R(\gamma,s)=H(\gamma,s)
=N(\gamma)\circ\bigl((1-s)\lambda_\gamma+s\id_\I\bigr).
\end{equation}
Corollary~\ref{cor:invariant-family} gives continuity in $E_Q$. Proposition~\ref{lem:normalization-identities} gives
\[
 R(\gamma,0)=\gamma,
 \qquad s>0\quad\Longrightarrow\quad R(\gamma,s)\text{ is regular}.
\]
\end{proof}

\begin{lemma}
\label{lem:regular-execution-extension}
Every continuous $b:\mathsf S^{n-1}\to E_Q$ admits an extension $\widetilde F:\mathsf D^n\to E_Q$ whose values on $\mathsf D^n\setminus\mathsf S^{n-1}$ are regular.
\end{lemma}
\begin{proof}
Use Proposition~\ref{lem:clock-regularization} to solve
\[
\begin{tikzcd}[column sep=large,row sep=large]
 \mathsf S^{n-1}\arrow[r,"b"]\arrow[d,hook]&E_Q\arrow[d]\\
 \mathsf D^n\arrow[r]\arrow[ur,dashed,"F"]&\{*\}.
\end{tikzcd}
\]
Choose a continuous $\epsilon_n:\mathsf D^n\to\I$ with $\epsilon_n^{-1}(0)=\mathsf S^{n-1}$, for example $\epsilon_n(z)=1-\|z\|$ for $n\geq1$ and $\epsilon_0=1$. Set
\begin{equation}
\label{eq:regular-globe-extension}
 \widetilde F(z)=R(F(z),\epsilon_n(z)).
\end{equation}
Then
\[
 \begin{gathered}
 \widetilde F|_{\mathsf S^{n-1}}=b,\\
 z\notin\mathsf S^{n-1}
 \quad\Longrightarrow\quad\epsilon_n(z)>0
 \quad\Longrightarrow\quad\widetilde F(z)\text{ is regular}.
 \end{gathered}
\]
\end{proof}

\begin{theorem}
\label{thm:regular-globular-lift}
For every q-cofibrant $\mathcal M$-multipointed $d$-space $A$, the following lifting problem admits a solution $f$ for which $\Sp(f):\Sp(A)\to\Cg$ is regular:
\[
\begin{tikzcd}[column sep=large,row sep=large]
 \varnothing\arrow[r]\arrow[d]&Q\arrow[d]\\
 A\arrow[r]\arrow[ur,dashed,"f"]&\mathbf1
\end{tikzcd}
\]
The object $Q$ can be any fixed cellular q-cofibrant replacement of $\mathbf1$.
\end{theorem}
\begin{proof}
For a cellular $A$, use the presentation of Lemma~\ref{lem:globular-presentation} and construct compatible $f_\nu:A_\nu\to Q$ with $\Sp(f_\nu)$ regular. The initial map sends every state to $*$ and is regular.

At a successor stage, the attaching map followed by $f_\nu$ determines $b:\mathsf S^{n-1}\to E_Q$. Lemma~\ref{lem:regular-execution-extension} gives $\widetilde F$. Evaluation defines $\widetilde f$ \cite[Proposition~13]{Moore3}, and the pushout gives $f_{\nu+1}$:
\[
\begin{tikzcd}[column sep=large]
 \Glob_{\mathcal M}(\mathsf{S}^{n-1})
 \arrow[r]\arrow[d,hook]
 & A_\nu\arrow[d]\arrow[ddr,bend left=20,"f_\nu"]&\\
 \Glob_{\mathcal M}(\mathsf{D}^n)
 \arrow[r]\arrow[drr,bend right=15,"\widetilde f"']&A_{\nu+1}\arrow[ul,phantom,"\ulcorner",very near start]\arrow[dr,dashed,"f_{\nu+1}"]&\\
 &&Q.
\end{tikzcd}
\]
The directed structure is generated by paths from $A_\nu$ and the new meridians \cite[Proposition~3.6 and Theorem~3.8]{GlobularNaturalSystem-full}. Boundary meridians factor through $A_\nu$, and interior meridians have nonconstant regular images. Lemma~\ref{lem:regular-generators} gives regularity at successors. At limits, use the induced map on the colimit and the same lemma.

For a retract $A\xrightarrow{i}B\xrightarrow{r}A$ of a cellular object, injectivity makes $\Sp(i)$ regular. Thus $A\xrightarrow{i}B\xrightarrow{f_B}Q$ has the required property.
\end{proof}

\begin{corollary}
\label{cor:globular-timed}
Let $\mathcal P\in\{\mathcal G,\mathcal M\}$. For every q-cofibrant $\mathcal P$-multipointed $d$-space $A$, the directed space $\Sp(A)$ admits a timed-space structure over the fixed globular clock $\Cg$. Consequently, for all $u,v\in|A|$, including points that are not states, the quotient
\[
 \dP(\Sp(A))(u,v)\longrightarrow\dT(\Sp(A))(u,v)
\]
is a trivial Hurewicz fibration.
\end{corollary}
\begin{proof}
For $\mathcal P=\mathcal M$, use Theorem~\ref{thm:regular-globular-lift}. For $\mathcal P=\mathcal G$, apply it to $A^{\mathcal M}$, which is q-cofibrant by Proposition~\ref{lem:change-reparametrizations}. Saturation follows from Theorem~\ref{thm:cofibrant-saturated-clock}, and the assertion about the quotient follows from Theorem~\ref{thm:timed-traces}.
\end{proof}

\begin{remark}
Every q-cofibrant realization is a saturated clock, but the maps of Theorem~\ref{thm:regular-globular-lift} place all of them in the single category $\Timed(\Cg)$. They depend on choices of extensions; formula~\eqref{eq:regular-globe-extension} ensures that those extensions are regular on interior meridians.
\end{remark}

\subsection{Comparing the cubical and globular clocks}

A finite cellular example distinguishes the globular and cubical clocks.

\begin{proposition}
\label{prop:no-cubical-clock}
There exists a finite cellular $\mathcal M$-multipointed $d$-space $A$ such that $\Sp(A)$ admits no regular directed map to the cubical clock $\dS$.
\end{proposition}
\begin{proof}
Start with a single state $v$, and attach $\Glob_{\mathcal M}(\mathsf{D}^0)$ with both endpoints identified with $v$. Denote the resulting multipointed $d$-space by $B$ and its characteristic execution loop by
\[
 a:\I\longrightarrow |B|.
\]
Thus $|B|$ is a circle. Attach one further globe by the pushout
\[
\begin{tikzcd}[column sep=large]
 \Glob_{\mathcal M}(\mathsf{S}^0)
 \arrow[r,"{(a,\,a*_N a)}"]\arrow[d,hook]
 &B\arrow[d]\\
 \Glob_{\mathcal M}(\mathsf{D}^1)
 \arrow[r]&A\arrow[ul,phantom,"\ulcorner",very near start].
\end{tikzcd}
\]
Both boundary meridians are execution paths with endpoints $v$, so $A$ is cellular and q-cofibrant.

The characteristic globe in Figure~\ref{fig:cubical-obstruction} gives
\begin{equation}
\label{eq:loop-homotopic-double}
 a\simeq a*_N a\qquad(\mathrm{rel}\ \{0,1\})
\end{equation}
in $|A|$. Cellular inclusions are embeddings \cite[Proposition~16]{Moore3}, so $a$ remains nonconstant.

Suppose $p:\Sp(A)\to\dS$ is regular. Then $c=pa$ is a nonconstant directed loop. For a nondecreasing lift,
\[
\begin{tikzcd}[column sep=large,row sep=large]
 \{0\}\arrow[r,"\widetilde c(0)"]\arrow[d,hook]&\mathbb R\arrow[d,two heads]\\
 \I\arrow[r,"c"']\arrow[ur,dashed,"\widetilde c"]&\mathbb R/\mathbb Z
\end{tikzcd}
\qquad
 n=\widetilde c(1)-\widetilde c(0)\in\mathbb Z_{>0}.
\]
Winding number identifies $\pi_1(\mathbb R/\mathbb Z)$ with $\mathbb Z$ \cite[Theorem~1.7 and its proof, pp.~29--31]{HatcherAlgebraicTopology-full}. Applying it to $c\simeq c*_N c$ gives
\[
n=\operatorname{wind}(c)=\operatorname{wind}(c*_N c)
=\operatorname{wind}(c)+\operatorname{wind}(c)=2n.
\]
Thus $n=0$, a contradiction.
\end{proof}

\begin{figure}[htbp]
\centering
\begin{tikzpicture}[x=1cm,y=1cm,>=stealth,font=\small]
  \fill[black!4]
    (0,0) .. controls (-1.65,.55) and (-1.65,3.05) .. (0,3.6)
    .. controls (1.65,3.05) and (1.65,.55) .. (0,0);
  \foreach \r in {-.68,0,.68} {
    \draw[dashed,black!45]
      (0,0) .. controls (\r,.65) and (\r,2.95) .. (0,3.6);
  }
  \draw[thick,->]
    (0,0) .. controls (-1.65,.55) and (-1.65,3.05) .. (0,3.6);
  \draw[thick,->]
    (0,0) .. controls (.825,.275) and (1.2375,1.0375) .. (1.2375,1.8);
  \draw[thick,->]
    (1.2375,1.8) .. controls (1.2375,2.5625) and (.825,3.325) .. (0,3.6);
  \fill (0,0) circle (1.7pt);
  \fill (0,3.6) circle (1.7pt);
  \fill (1.2375,1.8) circle (1.7pt);
  \node[below] at (0,-.04) {$0\mapsto v$};
  \node[above] at (0,3.64) {$1\mapsto v$};
  \node[left] at (-1.3,1.8) {$a$};
  \node[right] at (1.18,.85) {$a$};
  \node[right] at (1.18,2.75) {$a$};
  \node[right] at (1.28,1.8) {$\mapsto v$};
  \node[align=left,anchor=west,text width=6.1cm] at (3.1,1.8) {A regular map $p:\Sp(A)\to\dS$ would give
    \[
      \operatorname{wind}(p a)=n>0,
      \qquad n=2n.
    \]
  };
\end{tikzpicture}
\caption{The characteristic globe of the finite obstruction, shown before the boundary identifications. The marked points all map to the state $v$; the right boundary traverses $a$ twice. Interior meridians give the homotopy $a\simeq a*_N a$.}
\label{fig:cubical-obstruction}
\end{figure}

The globular clock nevertheless contains the cubical clock.

\begin{proposition}
\label{prop:common-globular-clock}
There is an injective morphism of directed spaces
\[
 j:\dS\longrightarrow\Cg.
\]
It is regular, and composition with $j$ defines a full and faithful functor
\[
 j_*:\Timed(\dS)\longrightarrow\Timed(\Cg).
\]
Consequently every standard directed realization of a precubical set admits a timed-space structure over the fixed globular clock.
\end{proposition}
\begin{proof}
Use the presentation of Lemma~\ref{lem:globular-presentation}. Since $E_Q\ne\varnothing$, there is a first globe. Its parameter disk is $\mathsf D^0$: for $n\geq1$, the attaching globe has execution paths and cannot map to the initial state space, which has none.

Both endpoints attach to $*$. For the resulting stage $B$,
\[
 \Sp(B)\cong\dS.
\]
The characteristic meridian is $t\mapsto[t]$. Its restrictions and nondecreasing reparametrizations have nondecreasing lifts; concatenation is obtained by translating the second lift by an integer. Conversely, a nondecreasing lift $\widetilde x$ has bounded image. Subdivision at its finitely many integer levels gives, on each piece $J_i$,
\[
 \widetilde x(J_i)\subseteq[k_i,k_i+1],
 \qquad x|_{J_i}(t)=[\widetilde x(t)-k_i].
\]
Thus these pieces are obtained from the meridian by continuous nondecreasing changes of parameter.

The inclusion $B\to Q$ is an embedding \cite[Proposition~16]{Moore3}, so $j:\dS\to\Cg$ is injective and regular. Composition with $j$ defines $j_*$.

For two timed spaces over $\dS$, $p:X\to\dS$ and $q:Y\to\dS$, a directed map $h:X\to Y$ satisfies
\[
 (j q)h=j p\quad\Longleftrightarrow\quad qh=p.
\]
Thus $j_*$ is full and faithful. Proposition~\ref{prop:cubical-timing} gives the required maps $j p_K:|K|\to\Cg$.
\end{proof}

\begin{corollary}
\label{cor:no-globular-to-cubical-clock}
There is no regular directed map $\Cg\to\dS$.
\end{corollary}
\begin{proof}
Let $A$ be the cellular object of Proposition~\ref{prop:no-cubical-clock}. By Theorem~\ref{thm:regular-globular-lift}, there is a regular directed map $\Sp(A)\to\Cg$. A regular map $\Cg\to\dS$ would yield a regular map $\Sp(A)\to\dS$ by composition, contradicting that proposition.
\end{proof}

\begin{figure}[htbp]
\centering
\begin{tikzcd}[row sep=3.4em,column sep=6em]
 {|K|}\arrow[r,"p_K"]\arrow[dr,"j p_K"']
   &\dS\arrow[d,hook,"j"]\\
 \Sp(A)\arrow[r,"p_A"']&\Cg
\end{tikzcd}
\qquad
\begin{tikzcd}[column sep=3.5em]
 \Timed(\dS)
   \arrow[r,hook,"j_*","\text{full and faithful}"']
   &\Timed(\Cg)
\end{tikzcd}
\caption{A common clock for cubical and globular directed topology. Here $K$ is any precubical set and $A$ is any q-cofibrant multipointed $d$-space. Every arrow in the first diagram is regular, and $j$ is injective.}
\label{fig:common-clock}
\end{figure}

\Needspace{6\baselineskip}
\subsection{A proof of the quotient theorem for execution paths}

For $\mathcal P\in\{\mathcal G,\mathcal M\}$, write
\[
 E_{\alpha,\beta}(A)=\mathbb P^{\mathrm{top}}_{\alpha,\beta}A,
 \qquad
 T_{\alpha,\beta}(A)=\mathbb P_{\alpha,\beta}A.
\]
The second space is the quotient of the first by the equivalence relation generated by $x\sim x\phi$, $\phi\in\mathcal P(1,1)$; it is the corresponding path space of the categorization flow $\mathrm{cat}_{\mathcal P}(A)$. No constant path is added when $\alpha=\beta$.

The next theorem proves the homotopy equivalence in \cite[Theorem~16]{Moore3} by normalization. For $\mathcal P=\mathcal G$, the trivial Hurewicz fibration assertion is recalled in \cite[pp.~101--102, before Theorem~16]{Moore3}. The same passage leaves open the existence of a section for arbitrary cellular $\mathcal M$-multipointed $d$-spaces; the regular map to $\Cg$ supplies a section and the homotopy lifting property.

\begin{theorem}
\label{thm:moore16}
Let $\mathcal P\in\{\mathcal G,\mathcal M\}$, let $A$ be a q-cofibrant $\mathcal P$-multipointed $d$-space, and let $\alpha,\beta\in A^0$. Then
\[
 E_{\alpha,\beta}(A)\longrightarrow T_{\alpha,\beta}(A)
\]
is a trivial Hurewicz fibration. Moreover, this quotient admits a continuous section and a strong deformation retraction onto the section's image that preserves each trace.
\end{theorem}

\begin{proof}
Write $E=E_{\alpha,\beta}(A)$, $T=T_{\alpha,\beta}(A)$, and $q:E\to T$. Choose $p:\Sp(A)\to\Cg$ by Corollary~\ref{cor:globular-timed}. By Lemma~\ref{lem:execution-normalization}, Corollary~\ref{cor:invariant-family} gives a section, a strong deformation retraction preserving traces, and a lift for every homotopy lifting problem:
\[
\begin{tikzcd}[column sep=large,row sep=large]
 Z\times\{0\}\arrow[r,"g"]\arrow[d,hook]&E\arrow[d,"q"]\\
 Z\times\I\arrow[r,"h"']\arrow[ur,dashed,"\widetilde h"]&T.
\end{tikzcd}
\]
\end{proof}

\begin{remark}
The proof uses the quotient topology of the family of execution paths directly. It does not require comparison with a subspace of the directed trace space. For $\mathcal G$, $\lambda_x$ is a homeomorphism; for $\mathcal M$, it may have nondegenerate fibers.
\end{remark}

\section{Saturated directed spaces without global clocks}
\label{sec:without-global-clocks}

We now apply the clock construction to compact metrizable subspaces determined by families of traces.

\begin{lemma}[Compact families of traces]
\label{lem:compact-trace-images}
Let $X$ be a directed space with Hausdorff underlying space. For every compact space $B$ and every continuous map $f:B\to\dT(X)(u,v)$, the subspace
\[
 K_f=\bigcup_{b\in B}\im(f(b))\subseteq |X|
\]
is compact. Here $\im(\tau)$ is the image of any representative of the trace $\tau$.
\end{lemma}
\begin{proof}
For open $U\subseteq|X|$, the quotient property gives
\[
\mathcal V_U=\{\tau\mid\im(\tau)\subseteq U\}\text{ open},
\qquad q_{u,v}^{-1}(\mathcal V_U)=[\I,U]\cap\dP(X)(u,v).
\]
Let $(U_i)_{i\in J}$ be a family of open subsets of $|X|$ covering $K_f$, and write $U_F=\bigcup_{i\in F}U_i$ for finite $F\subseteq J$. Compactness of each $\im(f(b))$ gives the open cover
\[
B=\bigcup_{\substack{F\subseteq J\\F\text{ finite}}}f^{-1}(\mathcal V_{U_F}).
\]
Compactness of $B$ gives finite $F_1,\ldots,F_m\subseteq J$ with
\[
B=\bigcup_{j=1}^m f^{-1}(\mathcal V_{U_{F_j}})
\quad\Longrightarrow\quad K_f\subseteq U_{F_1\cup\cdots\cup F_m}.
\]
Thus $K_f$ is quasicompact; it is Hausdorff as a subspace of $|X|$.
\end{proof}

\begin{lemma}[Restriction of a quotient]
\label{lem:closed-quotient-restriction}
Let $\pi:P\to Q$ be an ordinary quotient map and let $A\subseteq P$ be closed and a union of fibers of $\pi$. Then $\pi|_A:A\to\pi(A)$ is a quotient map for the relative topology on $\pi(A)$, and the induced map $A/{\sim}\to Q$ is a closed embedding. The following square is a pullback in ordinary topological spaces:
\[
\begin{tikzcd}[column sep=large,row sep=large]
A\arrow[r,hook]\arrow[d,two heads]\arrow[dr,phantom,"\lrcorner",very near start]&P\arrow[d,two heads,"\pi"]\\
A/{\sim}\arrow[r,hook]&Q.
\end{tikzcd}
\]
Here $\sim$ identifies points in the same fiber.
\end{lemma}
\begin{proof}
Since $\pi^{-1}(\pi(A))=A$, the subspace $\pi(A)$ is closed. The restriction is a quotient by \cite[Theorem~22.1]{zbMATH01461253}. The fiber relation gives $A/{\sim}\cong\pi(A)$, and the relative topology gives the displayed pullback. For corresponding embedding criteria in convenient categories, see \cite[Appendix~A, Lemma~7.2, p.~166]{Ref_wH} and \cite[Proposition~B.14]{leftproperflow-full}.
\end{proof}

\begin{lemma}[Paths in a subspace]
\label{lem:paths-coreflection}
Let $K$ be a subspace of a topological space $Z$ and let $\mathcal A$ be a family of paths with image in $K$. Its relative compact-open topology is the same whether the paths are regarded as maps to $Z$ or to $K$. If $Y=k_\Delta K$, the $\Delta$-kelleyfications of the relative compact-open topologies on $\mathcal A$ as maps to $K$ and to $Y$ are naturally homeomorphic.
\end{lemma}
\begin{proof}
For compact $J\subseteq\I$ and open $U\subseteq Z$,
\[
[J,U]\cap\mathcal A=[J,U\cap K]\cap\mathcal A,
\]
which identifies the two relative compact-open topologies. Put $A=k_\Delta\mathcal A$ using maps to $K$. By the exponential law \cite[Theorems~46.10--46.11]{zbMATH01461253}, Lemma~\ref{lem:interval-product}, and coreflection, evaluation factors as
\[
\begin{tikzcd}[column sep=large,row sep=large]
A\times\I\arrow[r,"\operatorname{ev}"]\arrow[dr,dashed,"\widetilde{\operatorname{ev}}"']&K\\
&Y\arrow[u].
\end{tikzcd}
\]
The adjoint of $\widetilde{\operatorname{ev}}$ is the inverse of the map on path families induced by $Y\to K$.
\end{proof}

\begin{proposition}
\label{prop:compact-trace-subspace}
Let $X$ be a Hausdorff saturated directed space. Let $K\subseteq|X|$ be compact metrizable and contain $u,v$, and equip $Y=k_\Delta K$ with the directed paths of $X$ contained in $K$. Put
\[
\begin{aligned}
P_X&=\dP_\co(X)(u,v),& P_K&=\{x\in P_X\mid x(\I)\subseteq K\},\\
Q_X&=P_X/{\sim},& Q_K&=P_K/{\sim},
\end{aligned}
\]
using the relative topology on $P_K$ and the ordinary quotient topologies on $Q_X$ and $Q_K$. Then $Y$ is a saturated clock, $Q_K$ is Hausdorff, and the induced map $\iota:Q_K\to Q_X$ is a closed embedding. Moreover,
\[
\dP(Y)(u,v)\cong k_\Delta P_K,
\qquad \dT(Y)(u,v)\cong k_\Delta Q_K.
\]
\end{proposition}
\begin{proof}
Coreflection identifies paths in $K$ with paths in $Y$. Thus saturation restricts from $X$ to $Y$, and the continuous injection $Y\to K$ makes $Y$ a clock. Since $K$ is closed,
\[
P_K=\bigcap_{t\in\I}\operatorname{ev}_t^{-1}(K)
\text{ is closed in }P_X,
\qquad \pi_X^{-1}(\pi_X(P_K))=P_K,
\]
where $\pi_X:P_X\to Q_X$ is the ordinary quotient. The second equality follows from invariance of path images under trace equivalence. Lemmas~\ref{lem:closed-quotient-restriction} and \ref{lem:paths-coreflection} give the closed embedding $\iota$ and $\dP(Y)(u,v)\cong k_\Delta P_K$.

Choose a metric on $K$. Its progress function and Lemmas~\ref{lem:factor-continuity} and~\ref{lem:constant-path-continuity}, with target $K$, give a continuous normalization $N:P_K\to P_K$ for the ordinary compact-open topology. Saturation of $Y$ gives directedness, and factorization preserves images and endpoints. Invariance under trace equivalence (Proposition~\ref{lem:normalization-identities}) gives
\[
\begin{tikzcd}[column sep=large]
P_K\arrow[r,two heads,"\pi_K"]\arrow[dr,"N"']&Q_K\arrow[d,"j"]\\
&P_K
\end{tikzcd}
\qquad j\pi_K=N,\qquad\pi_Kj=\id_{Q_K}.
\]
Consequently $Q_K\cong j(Q_K)\subseteq P_K$ is Hausdorff. The section identity also makes $k_\Delta\pi_K$ a quotient:
\[
(k_\Delta\pi_K)^{-1}(V)\text{ open}
\ \Longrightarrow\ 
V=(k_\Delta j)^{-1}\bigl((k_\Delta\pi_K)^{-1}(V)\bigr)\text{ open}.
\]
Its fibers are the trace classes, hence $\dT(Y)(u,v)\cong k_\Delta Q_K$.
\end{proof}

\Needspace{7\baselineskip}
\begin{proposition}
\label{lem:compact-trace-factorization}
Let $X$ be a saturated directed space with Hausdorff underlying space, and let $K\subseteq|X|$ be a compact metrizable subspace containing $u,v$. Let $Y=k_\Delta K$ be equipped with the directed paths of $X$ whose images are contained in $K$. Let $B$ be $\Delta$-generated. Every continuous map $f:B\to\dT(X)(u,v)$ such that $\im(f(b))\subseteq K$ for all $b\in B$ factors continuously through $\dT(Y)(u,v)$. If $B$ is also compact, then $f(B)$ is closed in $\dT(X)(u,v)$.
\end{proposition}
\begin{proof}
In the notation of Proposition~\ref{prop:compact-trace-subspace}, the map $k_\Delta P_X\to P_X$ induces a continuous bijection $r:\dT(X)(u,v)\to Q_X$. The hypothesis gives $rf(B)\subseteq\iota(Q_K)$, so the closed embedding $\iota$ gives the continuous factorization
\[
\begin{tikzcd}[column sep=large,row sep=large]
B\arrow[r,"f"]\arrow[d,"g=\iota^{-1}rf"']&\dT(X)(u,v)\arrow[d,"r"]\\
Q_K\arrow[r,hook,"\iota"']&Q_X.
\end{tikzcd}
\]
Coreflection gives $\widetilde f:B\to k_\Delta Q_K\cong\dT(Y)(u,v)$ with $i_T\widetilde f=f$, where $i_T$ is induced by $Y\to X$.

If $B$ is compact, then $g(B)$ is compact, hence closed in the Hausdorff space $Q_K$. Since $\iota$ is a closed embedding,
\[
\iota(g(B))\text{ closed in }Q_X
\quad\Longrightarrow\quad
f(B)=r^{-1}(\iota(g(B)))\text{ closed in }\dT(X)(u,v).
\]
\end{proof}

\begin{theorem}
\label{thm:metrizable-quasicompacts-traces}
For every Hausdorff saturated directed space $X$ such that every compact subspace of $|X|$ is metrizable, every quotient map from directed paths with fixed endpoints to traces is a trivial q-fibration of spaces. Its trace spaces are $\Delta$-Hausdorff.
\end{theorem}
\begin{proof}
For continuous $\gamma:\I\to\dT(X)(u,v)$, Lemma~\ref{lem:compact-trace-images} and the hypothesis make $K_\gamma$ compact metrizable. Proposition~\ref{lem:compact-trace-factorization} gives $\gamma(\I)$ closed. Thus the trace space is $\Delta$-Hausdorff.

Consider the lifting problem in the outer rectangle below. Lemma~\ref{lem:compact-trace-images} makes $K_f$ compact metrizable. Put $Y=k_\Delta K_f$ with the induced directed structure. Proposition~\ref{lem:compact-trace-factorization} gives $f_K$. Since $q_{u,v}a=f|_{\mathsf S^{n-1}}$, every $a(z)$ has image in $K_f$; Proposition~\ref{prop:compact-trace-subspace} and coreflection give $a_K$.
\[
\begin{tikzcd}[column sep=large,row sep=large]
\mathsf S^{n-1}\arrow[r,"a_K"]\arrow[d,hook]
&\dP(Y)(u,v)\arrow[r,"i_P"]\arrow[d,"q_Y"]\arrow[dr,phantom,"\lrcorner",very near start]
&\dP(X)(u,v)\arrow[d,"q_{u,v}"]\\
\mathsf D^n\arrow[r,"f_K"']\arrow[ur,dashed,"\ell_K"]
&\dT(Y)(u,v)\arrow[r,"i_T"']
&\dT(X)(u,v).
\end{tikzcd}
\]
The right square is a pullback by Lemma~\ref{lem:paths-coreflection}. Corollary~\ref{cor:submetrizable} makes $q_Y$ a trivial q-fibration, giving $\ell_K$. Therefore
\[
\ell=i_P\ell_K,
\qquad \ell|_{\mathsf S^{n-1}}=i_Pa_K=a,
\qquad q_{u,v}\ell=i_Tq_Y\ell_K=i_Tf_K=f.
\]
This proves the lifting criterion for every $n\geq0$.
\end{proof}

\Needspace{9\baselineskip}
\begin{theorem}
\label{thm:second-countable-traces}
For every Hausdorff second countable saturated directed space $X$, every quotient map from directed paths with fixed endpoints to traces is a trivial q-fibration of spaces. Its trace spaces are $\Delta$-Hausdorff.
\end{theorem}
\begin{proof}
Every compact subspace of $|X|$ is regular and second countable, hence metrizable by \cite[Theorem~34.1]{zbMATH01461253}. Apply Theorem~\ref{thm:metrizable-quasicompacts-traces}.
\end{proof}

\begin{lemma}
\label{lem:metrizable-images}
A Hausdorff continuous image of a compact metrizable space is compact metrizable.
\end{lemma}
\begin{proof}
Let $f:K\to L$ be the surjection. Compactness makes $f$ closed and $L$ compact \cite[\S26]{zbMATH01461253}. Let $\mathcal U$ be the countable family of finite unions of elements of a countable base of $K$. Then
\[
W_U=L\setminus f(K\setminus U),\qquad U\in\mathcal U,
\]
is open. For open $V\subseteq L$ and $y\in V$, compactness of $f^{-1}(y)$ gives $U\in\mathcal U$ such that
\[
f^{-1}(y)\subseteq U\subseteq f^{-1}(V)
\quad\Longrightarrow\quad y\in W_U\subseteq V.
\]
Thus $(W_U)_{U\in\mathcal U}$ is a countable base. Apply Urysohn's metrization theorem \cite[Theorem~34.1]{zbMATH01461253}.
\end{proof}

\begin{remark}
\label{rem:compact-cellular-comparison}
Theorem~\ref{thm:metrizable-quasicompacts-traces} also applies to the standard directed realization of every precubical set and to $\Sp(A)$ for every cellular $\mathcal P$-multipointed $d$-space $A$, where $\mathcal P\in\{\mathcal G,\mathcal M\}$. Both underlying spaces are Hausdorff, and every compact subset meets only finitely many cells, including states as cells of dimension $0$; see \cite[Proposition~1.5.2]{MR1074175} and \cite[Propositions~15--16]{Moore3}. A compact subset $K$ is therefore contained in $F=\bigcup_{j=1}^m\overline e_j$. The characteristic maps give a surjection $\coprod_{j=1}^m\mathsf D^{n_j}\to F$, so Lemma~\ref{lem:metrizable-images} makes $F$, and hence $K$, metrizable. Saturation follows from Lemma~\ref{lem:cubical-saturated} and Theorem~\ref{thm:cofibrant-saturated-clock}.

For these realizations, the clock maps of Proposition~\ref{prop:cubical-timing} and Corollary~\ref{cor:globular-timed} give the stronger conclusion of a trivial Hurewicz fibration. The argument using compact families alone does not prove the homotopy lifting property for an arbitrary parameter space $Z$, since the image of $Z\times\I$ in the trace space need not be compact.

\end{remark}

\section{Functorial normalization}
\label{sec:functorial-normalization}

Fix a clock $C$, a continuous injective map $\iota:|C|\to(M,d)$ into a metric space, and a triple $\mathfrak a$ as in Notation~\ref{not:multiscale-profile}. We use the normalization and homotopies of Section~\ref{sec:normalization}.

\begin{proposition}[Naturality over a fixed clock]
\label{prop:natural-normalization}
If $h:(X,p)\to(Y,q)$ is a morphism in $\Timed(C)$, then
\[
h\circ N_p(x)=N_q(h\circ x),\qquad
h\circ H_p(x,s)=H_q(h\circ x,s).
\]
The lifts obtained from Proposition~\ref{prop:abstract-hurewicz} using \eqref{eq:homotopy-lift-parameter}--\eqref{eq:homotopy-lift} are also natural for these morphisms. Subscripts on $j$ and $K$ indicate the timed space used; $h_*$ denotes the induced map on paths or traces.
\[
\begin{tikzcd}[column sep=large,row sep=large]
 \dP(X)(u,v)\arrow[r,"N_p"]\arrow[d,"h_*"']&\dP(X)(u,v)\arrow[d,"h_*"]\\
 \dP(Y)(h(u),h(v))\arrow[r,"N_q"']&\dP(Y)(h(u),h(v)).
\end{tikzcd}
\]
\end{proposition}
\begin{proof}
Since $qh=p$ and $h$ is regular (Theorem~\ref{thm:timed-locally-presentable}),
\[
\ell_{h\circ x}=L_{\iota\circ q\circ h\circ x}=L_{\iota\circ p\circ x}=\ell_x,
\qquad \lambda_{h\circ x}=\lambda_x.
\]
Moreover, $\eta_{h\circ x,s}=\eta_{x,s}$ whenever $s+(1-s)\ell_x(1)>0$. Uniqueness in \eqref{eq:timed-factorization} and surjectivity of the quotients give
\[
h\circ N_p(x)=N_q(h\circ x),\qquad h\circ j_p(\tau)=j_q(h_*\tau).
\]
Hence
\[
\begin{aligned}
h\circ H_p(x,s)
&=h\circ N_p(x)\circ\bigl((1-s)\lambda_x+s\id_\I\bigr)
=H_q(h\circ x,s),\\
h\circ K_p(x,\tau,s)
&=j_q(h_*\tau)\circ\eta_{h\circ x,s}
=K_q(h\circ x,h_*\tau,s).
\end{aligned}
\]
At $s=0$ with $x$ constant, the last equality follows from the definition of $K$. Formula~\eqref{eq:abstract-homotopy-lift} then gives naturality of the homotopy lifts.
\end{proof}

\begin{remark}[Natural splitting]
\label{rem:natural-splitting}
Let $\mathcal B_C$ be the category whose objects are triples $(p:X\to C,u,v)$ with $u,v\in|X|$ and whose morphisms are morphisms in $\Timed(C)$ preserving the two chosen points. The assignments
\[
\mathsf P(p,u,v)=\dP(X)(u,v),
\qquad \mathsf T(p,u,v)=\dT(X)(u,v)
\]
define functors $\mathsf P,\mathsf T:\mathcal B_C\to\Top$. By Theorem~\ref{thm:timed-traces} and Proposition~\ref{prop:natural-normalization}, the quotient maps and their normalization sections define natural transformations
\[
\begin{tikzcd}[column sep=large]
\mathsf P\arrow[r,shift left=1ex,"q"]
&\mathsf T\arrow[l,shift left=1ex,"j"]
\end{tikzcd}
\qquad qj=\id_{\mathsf T},\qquad jq=N,\qquad N^2=N.
\]
Thus $\mathsf T$ splits the idempotent natural transformation $N$ of $\mathsf P$. Moreover, the homotopy $H:\mathsf P\times\I\to\mathsf P$ is natural and satisfies
\[
H(-,0)=\id_{\mathsf P},\qquad H(-,1)=jq,
\qquad qH(x,s)=q(x),\qquad H(j(\tau),s)=j(\tau).
\]
Thus $j$ identifies the trace functor with a natural strong deformation retract of the path functor.
\end{remark}

\begin{remark}[The category of normalized paths]
\label{rem:normalized-category}
For a timed space $p:X\to C$, let $\Nat_p(X)$ have the points of $|X|$ as objects and the spaces $\Nat_p(X)(u,v)$ as morphism spaces. For normalized paths $a:u\to v$ and $b:v\to w$, define
\[
a\star b=N_p(a*_N b).
\]
Concatenation and normalization are continuous, so this composition is continuous for products in $\Top$. Compatibility of trace equivalence with concatenation (Lemma~\ref{lem:trace-concatenation}) and invariance of normalization (Proposition~\ref{lem:normalization-identities}) give
\[
\begin{aligned}
(a\star b)\star c
&=N_p\bigl(N_p(a*_N b)*_N c\bigr)
=N_p\bigl((a*_N b)*_N c\bigr)\\
&=N_p\bigl(a*_N(b*_N c)\bigr)
=N_p\bigl(a*_N N_p(b*_N c)\bigr)
=a\star(b\star c).
\end{aligned}
\]
Here the two parenthesizations of $*_N$ differ by an increasing piecewise linear reparametrization. The constant paths are identities, since $c_u*_N a\sim a\sim a*_N c_v$. Thus $\Nat_p(X)$ is a category enriched in $\Top$. Give the trace spaces the composition $[a][b]=[a*_N b]$. The homeomorphisms of Theorem~\ref{thm:timed-traces} preserve composition and identities, and therefore identify $\Nat_p(X)$ with the trace category $\dT(X)$ as categories enriched in $\Top$.

For a morphism $h:(X,p)\to(Y,q)$, Proposition~\ref{prop:natural-normalization} gives
\[
h(a\star b)=N_q\bigl(h(a)*_N h(b)\bigr)=h(a)\star h(b).
\]
Consequently normalization defines a functor from $\Timed(C)$ to categories enriched in $\Top$, naturally isomorphic to the functor assigning trace categories. This construction does not require additivity of the progress function.
\end{remark}

\begin{remark}[The multiplicative monoid of the interval]
\label{rem:interval-monoid-action}
Fix $p:X\to C$ and endpoints $u,v$, and write $P=\dP(X)(u,v)$ and $H_s(x)=H_p(x,s)$. For $x\in P$, put $\sigma_s=(1-s)\lambda_x+s\id_\I$. For nonconstant $x$, put $L=\ell_x(1)>0$. Equivariance and Proposition~\ref{lem:normalization-identities} give
\[
\begin{aligned}
N_p(H_s(x))&=N_p(x),\\
\ell_{H_s(x)}(t)&=\ell_{N_p(x)}(\sigma_s(t))=L\sigma_s(t),\\
\lambda_{H_s(x)}(t)&=\frac{L\sigma_s(t)}{L\sigma_s(1)}=\sigma_s(t).
\end{aligned}
\]
For constant $x$, the first and last identities follow from $H_s(x)=x$ and $\lambda_x=\id_\I$. Consequently
\[
\begin{aligned}
H_t(H_s(x))
&=N_p(x)\circ\bigl((1-t)((1-s)\lambda_x+s\id_\I)+t\id_\I\bigr)\\
&=N_p(x)\circ\bigl((1-s)(1-t)\lambda_x+(s+t-st)\id_\I\bigr)
=H_{s+t-st}(x).
\end{aligned}
\]
Set $D_a(x)=D(x,a)=H_p(x,1-a)$ and $\mu(a,b)=ab$. Then
\[
D_aD_b=D_{ab},\qquad D_1=\id_P,\qquad D_0=N_p,
\]
so $D$ is a continuous action of the topological monoid $(\I,\mu,1)$:
\[
\begin{tikzcd}[column sep=large,row sep=large]
P\times\I\times\I\arrow[r,"D\times\id_\I"]\arrow[d,"\id_P\times\mu"']
&P\times\I\arrow[d,"D"]\\
P\times\I\arrow[r,"D"']&P.
\end{tikzcd}
\]
The fixed points are
\[
\{x\in P\mid D_a(x)=x\ \forall a\in\I\}
=\{x\in P\mid N_p(x)=x\}=\Nat_p(X)(u,v).
\]
Proposition~\ref{prop:natural-normalization} makes maps induced by morphisms over $C$ equivariant.
\end{remark}

%\bibliographystyle{plainurlwithoutprefixDOI}
%\bibliography{Bibliotheque,timed-space-full-bibliography}

\end{document}